\documentclass{amsart}
\usepackage{hyperref}
\hypersetup{
    colorlinks=true, 
    linkcolor=blue,  
    citecolor=blue,  
    urlcolor=blue    
}
\usepackage{mathtools}
\usepackage[numbers]{natbib} 
\usepackage{amssymb}

\usepackage{graphicx}

\usepackage{xcolor}
\usepackage{etoolbox}

\newcommand{\Hess}{D^2}

\newcommand{\R}{\mathbb R}

\newcommand{\Lop}{L_{\R^n}}
\newcommand{\LopOneD}{\tilde{L}}
\newcommand{\Psif}{\Psi}

\newcommand{\PV}{\operatorname{PV}}

\numberwithin{equation}{section}
\theoremstyle{plain} 
\newtheorem{thm}[equation]{Theorem}

\newtheorem{lem}[equation]{Lemma}
\newtheorem{prop}[equation]{Proposition}
\theoremstyle{definition}

\theoremstyle{remark}
\newtheorem{rem}[equation]{Remark}

\newtheorem{conj}[equation]{Conjecture}

\begin{document}
\title[Modulus of continuity]{Modulus of continuity for solutions of non-local heat equations}
\author{Ben Andrews}
\author{Sophie Chen}
\date{\today}

\begin{abstract}
We extend the two-point modulus-of-continuity method for parabolic
equations---as used, for instance, in the proof of the fundamental gap
conjecture---to non-local heat equations on $\R^n$. An
initial modulus of continuity is shown to be preserved provided that
its odd extension is a supersolution of an associated one-dimensional
non-local equation. The result applies to both integrable and singular
kernels, including the fractional Laplacian, and the method
further applies to a class of nonlinear non-local heat equations on $\R^n$.

We also examine the corresponding regional problem on bounded convex sets. The natural distance-only modulus comparison fails: a smooth counterexample is constructed in one dimension, and
the same obstruction is demonstrated to persist in higher dimensions.
For the regional fractional Laplacian, we define convex domains of fixed
diameter for which the first non-trivial eigenvalue
tends to zero. There is consequently no positive
Payne--Weinberger-type lower bound for this eigenvalue depending only
on the diameter of the domain.
\end{abstract}

\let\thefootnote\relax
\makeatletter\def\Hy@Warning#1{}\makeatother

\maketitle
\tableofcontents
\section{Introduction and main results}\label{ch3_intro_main_results}
Let $L_\Omega$ denote the non-local linear operator
\begin{align}\label{def_non_local_integral_operator}
    L_\Omega u(x):=\int_{\Omega}\rho(z-x)(u(z)-u(x))\,dz,
    \qquad x\in\Omega,
\end{align}
where $u$ belongs to a suitable function space, $\Omega$ is an open
set in $\R^n$, and the kernel $\rho$ is non-negative, rotationally
symmetric, non-increasing as a function of the distance from the
origin, and decays to zero at infinity. We consider solutions
$u(x,t)$ of the corresponding non-local heat equation
\begin{equation}\label{def_non_local_heat_equation}
    \begin{cases}
        u_t(x,t)=L_\Omega u(x,t),
        & (x,t)\in\Omega\times(0,\infty),\\
        u_0(x)=u(x,0),
        & x\in\Omega.
    \end{cases}
\end{equation}
Set $R_\Omega=\frac{\operatorname{diam}(\Omega)}{2}$, allowing $R_\Omega=\infty$, and let
\[
    J_\Omega=
    \begin{cases}
        [0,R_\Omega], & R_\Omega<\infty,\\
        [0,\infty),   & R_\Omega=\infty.
    \end{cases}
\]
Recall that $\omega:J_\Omega\to[0,\infty)$ is a
\emph{modulus of continuity} for $f:\Omega\to\R$ if
$\omega(0)=0$, $\omega$ is continuous at $0$, and
\[
    |f(y)-f(x)|
    \leq
    2\omega\!\left(\frac{|y-x|}{2}\right)
\]
for all $x,y\in\Omega$. Our aim is to determine conditions under which an initial modulus of
continuity for a solution of \eqref{def_non_local_heat_equation} is
preserved at all subsequent times. 

Three cases are of interest:
\begin{itemize}
    \item[1.] $\Omega=\R^n$ with no boundary condition;
    \item[2.] $\Omega\subset\R^n$ is a bounded convex domain, giving
    rise to the regional heat equation; and
    \item[3.] $\Omega=\R^n$ with no boundary condition, but with a
    non-linearity added to the non-local operator.
\end{itemize}
These choices were made on the basis of their potential future applications to prove, for instance, a non-local version of the fundamental gap conjecture. In the first case, we prove modulus of continuity theorems for solutions of the non-local heat equation on $\R^n$ (Section \ref{sec_fractionalheat_Rn}). It turns out, however, that for our choice of non-local Neumann boundary condition, the modulus of continuity conjecture fails in all dimensions. In particular, this implies that extending the classical Payne-Weinberger inequality to the nonlocal regional operator is not possible. We construct counterexamples in these situations. (Section \ref{sec_regional_heat_equation}). Finally, the third non-linear case was included as a natural extension of the techniques employed in the proof of the first case (Section \ref{sec_non_linear_non_local_heat_eqn}).

The two-point modulus-of-continuity method used here was developed by
Andrews and Clutterbuck for quasilinear parabolic equations, leading
to sharp gradient estimates and applications to eigenvalue and
spectral-gap bounds \cite{AC,AC1,AC2,AC3}, and was subsequently
extended by Li to viscosity solutions of fully nonlinear parabolic equations \cite{LiFullyNonlinear}. In the non-local setting, preserved
moduli of continuity form the basis of the non-local maximum principle
of Kiselev, Nazarov and Volberg for the critical surface
quasi-geostrophic equation \cite{KNV}. Gradient estimates for
L\'evy semigroups and transition densities have also been obtained by
different, probabilistic and heat-kernel methods \cite{SchillingSztonykWang,KulczyckiRyznar}. Our present work adapts the Andrews--Clutterbuck two-point comparison method directly
to non-local heat evolutions, with the controlling modulus governed
by a one-dimensional non-local equation.

We now turn to the main results of this paper, starting with a modulus of continuity theorem for solutions of a non-local heat equation on $\R^n$.

\section{Non-local heat equation on $\R^n$}\label{sec_fractionalheat_Rn}

Consider sufficiently smooth solutions $u:\R^n\times[0,\infty)\rightarrow\R$ of the non-local heat equation
\begin{equation}\label{eq:nonlocal_heat}
\begin{cases}
u_{t}(x,t)=L_{\R^n} u(x,t), & (x,t)\in\R^n\times(0,\infty) \\
u(x,0)=u_{0}(x), & x\in\R^n,
\end{cases}
\end{equation}
where the non-local linear operator $L_{\R^n}$ is defined by
\begin{equation}
L_{\R^n} u(x):=\int_{\R^n}\rho(z-x)(u(z)-u(x))dz.
\end{equation}

\subsection{Lebesgue integrable kernels}
First we address the class of non-local operators whose kernel $\rho$ belongs to $L^1(\R^n)$ and also satisfy the properties listed in Section \ref{ch3_intro_main_results}.

\begin{thm}[Non-local modulus of continuity on $\R^n$]\label{thm_modulus_Rn}
Suppose $u:\R^n\times[0,\infty)\rightarrow\R$ satisfies the non-local heat equation \eqref{eq:nonlocal_heat} and $\|u\|_{L^{\infty}(\R^n)}\le 1$. Let $\varphi:[0,\infty)\times[0,\infty)\rightarrow[0,\infty)$ be a function with the following properties:
\begin{enumerate}
    \item[(a.)] $\varphi':= \partial\varphi(r,t)/\partial r>0$ and $\lim_{r\rightarrow0^{+}}\varphi(r,t)=\varphi(0,t)=0$ for each $t\ge0$.
    \item[(b.)] $\varphi(\cdot,0)$ is a modulus of continuity for $u(\cdot,0)$.
    \item[(c.)] $\varphi$ has an odd extension
    \begin{equation}\label{eq:odd_extension}
    \tilde{\varphi}(r,t)=\begin{cases}\varphi(r,t),&r\ge0\\ -\varphi(-r,t),&r<0\end{cases}
    \end{equation}
    satisfying $\lim_{r\rightarrow\pm\infty}\tilde{\varphi}(r,t)=\pm1$ and being a supersolution of the one-dimensional non-local heat equation
    \begin{equation}\label{eq:1d_heat_eq}
    \tilde{\varphi}_{t}(r,t)\geq\int_{\R}\tilde{\rho}(w)(\tilde{\varphi}(r+w,t)-\tilde{\varphi}(r,t))dw,\quad r\in\R^+,t>0,
    \end{equation}
    where $\tilde{\rho}(w)=\int_{\R^{n-1}}\rho(w,p)dp$.
\end{enumerate}
Furthermore, assume we can define a suitable regularisation function:
\begin{enumerate}
    \item[(d.)] There exists a smooth, radially symmetric regularisation function $\Psi:\R^n \to [1,\infty)$ such that $\Psi(x) \to \infty$ as $|x|\to\infty$, $L_{\R^n}\Psi(x)$ is well-defined for every $x\in\R^n$, and the action of the operator $L_{\R^n}$ on $\Psi$ is globally bounded:
    \begin{equation}\label{eq:K_bound}
        0\leq K := \sup_{x\in\R^n} L_{\R^n}\Psi(x) < \infty.
    \end{equation}
\end{enumerate}
Then $\varphi(\cdot,t)$ is a modulus of continuity for $u(\cdot,t)$ on $\R^n\times[0,\infty)$.
\end{thm}


\begin{proof}
We want to show that $|u(y,t)-u(x,t)|\le 2\varphi\left(\frac{|y-x|}{2},t\right)$ for all $x, y \in \R^n$ and $t \ge 0$.\\

\textbf{Step 1: Define the auxiliary function.}
Let $K$ be the constant from condition (d). Choose a constant $C > K$. We define an auxiliary function $Z_{\epsilon}:\R^n\times\R^n\times[0,\infty)\rightarrow\R$ using the regularisation function $\Psi(x)$:
\begin{equation}
Z_{\epsilon}(x,y,t)=u(y,t)-u(x,t)-2\varphi\left(\frac{|y-x|}{2},t\right)-\epsilon e^{Ct}(\Psi(x)+\Psi(y)).
\end{equation}
It suffices to show that $Z_{\epsilon}<0$ for each $\epsilon>0$.

The cases $x=y$ and $t=0$ follow from conditions (a.) and (b.), respectively. Since $\|u\|_\infty \le 1$ and $\Psi(x) \to \infty$ as $|x|\to\infty$, $Z_\epsilon \to -\infty$ as $|x|$ or $|y| \to \infty$. Therefore, $Z_\epsilon$ must attain a global maximum at some finite point.\\

\textbf{Step 2: Maximum principle and time derivative.}
Suppose for contradiction that for some $\epsilon_0>0$, $Z_{\epsilon_0} \ge 0$ at some point. Then there exists a first time $t_{0}>0$ and distinct points $x_{0}, y_{0} \in \R^n$ such that $Z_{\epsilon_0}(x_{0},y_{0},t_{0})=0$, and $Z_{\epsilon_0}(a,b,t)\leq 0$ globally for $t\leq t_0$.

Simplify notation: let $\epsilon=\epsilon_0$, $x=x_{0}$, $y=y_{0}$, $t=t_0$, $s=|x-y|/2$, $E = \epsilon e^{Ct}$, and suppress time dependence in $u$ and $\varphi$.

At the maximum point $(x,y,t)$, the time derivative satisfies:
\begin{align}\label{eqn_Z_time_derivative2_n}
0 \le \frac{\partial Z_{\epsilon}}{\partial t} &= u_{t}(y)-u_{t}(x)-2\varphi_{t}(s) - C E (\Psi(x)+\Psi(y)) \nonumber \\
&=\int_{\R^n}\rho(z-y)(u(z)-u(y))dz-\int_{\R^n}\rho(z-x)(u(z)-u(x))dz\\
&- 2\varphi_{t}(s) - C E (\Psi(x)+\Psi(y))\nonumber.
\end{align}

\textbf{Step 3: Coupling-by-reflection and maximum principle bounds.}
We wish to generate a contradiction to inequality (\ref{eqn_Z_time_derivative2_n}). Simply rearranging the integrands and applying 
the maximum principle to the result does not work as all terms cancel except $-2\varphi_t$. To avoid this problem, our strategy is to introduce $\varphi_t$ into the integrals via condition (c), which will enable us to cancel it out. 

First, we make a change of variables $z=r+x$ (or $z=r+y$), and choose coordinates so that the first standard unit vector in $\R^n$ is given by $e_1=\frac{y-x}{|y-x|}$. Then for any vector $r$ in $\R^n$ relative to a reference point $x$, its reflection in the hyperplane through $x$ and perpendicular to the chord $\overline{xy}$ joining $x$ and $y$ is given by
\begin{align}\label{eqn_reflctn}
    Rr=r-2\langle r, e_1\rangle e_1,
\end{align}
where $\langle\cdot,\cdot\rangle$ denotes dot product. Since the kernel $\rho$ only depends on the distance between two points, it follows that $\rho(r)=\rho(|r|)=\rho(|Rr|)=\rho(Rr)$, so we can rewrite inequality (\ref{eqn_Z_time_derivative2_n}) using the relationship
\begin{align*}
    \int_{\R^n}\rho(r)u(y+r)dr=\int_{\R^n}\rho(r)u(y+Rr)dr
\end{align*}
to give
\begin{align}\label{eqn_reflctn_integ}
    0\leq &\int_{\R^n}\rho(r)(u(y+Rr)-u(x+r))dr-\int_{\R^n}\rho(r)(u(y)-u(x))dr\nonumber\\
    &-2\varphi_t(s)-C E (\Psi(x)+\Psi(y)).
\end{align}

Next, observe that given a vector $r$ such that $z=x+r$, there exists a vector $r'$ joining $x$ to the reflection of $z$ in the hyperplane perpendicular to the mid-point of the chord $\overline{xy}$ so that 
\begin{align}\label{eqn_geom_changevar}
    x+r&=y+Rr'\qquad\mbox{and}\qquad y+Rr=x+r'.
\end{align}
See Figure \ref{reflection}. 
\begin{figure}[htp]
    \centering
    \includegraphics[width=12cm]{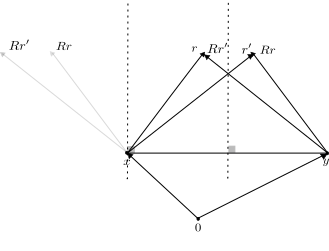}
    \caption{Change of variables using a `coupling-by-reflection' technique.}\label{reflection}
\end{figure}
Moreover, we can partition $\R^n$ into three mutually disjoint sets
\begin{align*}
    S&:=\left\{r\in\R^n:\langle r,e_1\rangle =\frac{|x-y|}{2}\right\}\\
    R&:=\left\{r\in\R^n:\langle r,e_1\rangle >\frac{|x-y|}{2}\right\}\\
    L&:=\left\{r\in\R^n:\langle r,e_1\rangle <\frac{|x-y|}{2}\right\}.
\end{align*}
Using the equations of  (\ref{eqn_geom_changevar}) in the first integral of inequality (\ref{eqn_reflctn_integ}), evaluated over the integration region $S$, implies that $r=r'$; thus,
\begin{align*}
    \int_{S}\rho(r)(u(y+Rr)-u(x+r))dr&=\int_{S}\rho(r)(u(x+r')-u(x+r))dr\\
    &=\int_{S}\rho(r)(u(x+r)-u(x+r))dr\\
    &=0.
\end{align*}
Otherwise, $r\neq r'$, and the vectors belong to different half-planes $R$ and $L$. To see this, suppose that $r$ belongs to $L$. Using equations (\ref{eqn_geom_changevar}) and (\ref{eqn_reflctn}), and writing $r_1=\langle r, e_1\rangle$ and $r'_1=\langle r', e_1\rangle$, we have
\begin{align*}
    r'=-(y-x)+r+2r'_1e_1,
\end{align*}
which implies
\begin{align*}
    r'_1=-|y-x|+r_1+2r'_1,
\end{align*}
or
\begin{align*}
    r'_1=|y-x|-r_1>\frac{|x-y|}{2},
\end{align*}
since $r_1<\frac{|x-y|}{2}$. Hence, $r'$ belongs to $R$. (Of course, we can interchange the roles of $r$ and $r'$, and $R$ and $L$, in the above argument.) Thus, inequality (\ref{eqn_reflctn_integ}) becomes
\begin{align}\label{eqn_reflctn_integ2}
    0\leq &\int_L\rho(r)(u(y+Rr)-u(x+r))dr+\int_R\rho(r)(u(y+Rr)-u(x+r))dr\nonumber\\
    &-\int_{\R^n}\rho(r)(u(y)-u(x))dr-2\varphi_t(s)- C E (\Psi(x)+\Psi(y)).
\end{align}

Now given a vector $r$ in $R$, define a map $g: r'\mapsto r$, where $r'$ is in $L$, so that using equations (\ref{eqn_geom_changevar}), we can change variables $r=y-x+Rr'$ and re-write the integrand in the second integral of inequality (\ref{eqn_reflctn_integ2}) as
\begin{align*}
    \rho(r)(u(y+Rr)-u(x+r))=\rho(g(r'))(u(x+r')-u(y+Rr')).
\end{align*}
This allows us to combine the first two integrals over a common integration region $L$ giving
\begin{align*}
    0\leq&\int_L(\rho(r)-\rho(g(r)))(u(y+Rr)-u(x+r))dr-\int_{\R^n}\rho(r)(u(y)-u(x))dr\\
    &-2\varphi_t(s)- C E (\Psi(x)+\Psi(y)).
\end{align*}
Observe that since $r\in L$, $g(r)\in R$, and $\rho$ is a non-negative, non-increasing, rotationally symmetric kernel, $\rho(r)\geq\rho(g(r))$, so $\rho(r)-\rho(g(r))\geq0$. Hence, we may apply the maximum principle
to both integrals yielding

\begin{align*}
    0\leq&\int_L(\rho(r)-\rho(g(r)))\left(2\varphi\left(\frac{|y-x|}{2}-r_1\right)+E (\Psi(x+r)+\Psi(y+Rr)))\right)dr\\
    &-\int_{\R^n}\rho(r)\left(2\varphi\left(\frac{|x-y|}{2}\right)+E (\Psi(x)+\Psi(y)))\right)dr-2\varphi_t(s)- C E (\Psi(x)+\Psi(y)).
\end{align*}

    Finally, to introduce $\varphi_t$ into the calculation, we need to split up the first
integral into an integral over $L$ and an integral over $R$. We thus perform another change of variables $r=y-x+Rr'$, again using (\ref{eqn_geom_changevar}), and calculate
\begin{align*}
    \frac{|x-y|}{2}-r_1&=\frac{|y-x|}{2}-\langle r,e_1\rangle\\
    &=\frac{|y-x|}{2}-\langle y-x+Rr',e_1 \rangle\\
    &=-\frac{|y-x|}{2}+r_1'.
\end{align*}
This gives
\begin{align*}
0 &\leq\int_L\rho(r)\left(2\varphi\left(\frac{|y-x|}{2}-r_1\right)+E (\Psi(x+r)+\Psi(y+Rr)))\right)dr\\
&-\int_R\rho(r)\left(2\varphi\left(-\left(\frac{|y-x|}{2}-r_1\right)\right)+E (\Psi(x+r)+\Psi(y+Rr)))\right)dr\\
&-\int_{\R^n}\rho(r)\left(2\varphi\left(\frac{|y-x|}{2}\right)+E (\Psi(x)+\Psi(y)))\right)dr
-2\varphi_t(s)-CE(\Psi(x)+\Psi(y))\\
&\leq 2\int_{\R^n}\rho(r)\left(\tilde\varphi\left(\frac{|y-x|}{2}-r_1\right)-\tilde\varphi\left(\frac{|y-x|}{2}\right)\right)dr\\
&+E\int_{\R^n}\rho(r)\left(\Psi(y+Rr)-\Psi(y) +\Psi(x+r)-\Psi(x)\right)dr\\
&-2\varphi_t(s)-CE(\Psi(x)+\Psi(y)),
\end{align*}
Here equation~(\ref{eq:odd_extension}), defining the odd extension,
combines the $L$- and $R$-contributions involving $\varphi$. For the
$\Psi$-terms, the assumption $\Psi\geq 1$ allows us to bound the
$L$- and $R$-contributions by the corresponding full-space integral.

Let $I_\varphi$ and $I_\Psi$ denote the first and second integrals, respectively. We analyse each integral in turn.\\

\textbf{Step 4: Analyse the terms.} To resolve the first term $I_\varphi$, integrate out $\R^{n-1}$ dimensions,  apply a change of variables, and use condition (c.):
\begin{align*}
    I_\varphi&=2\int_{\R^n}\rho(r)\left(\tilde\varphi\left(\frac{|y-x|}{2}-r_1\right)-\tilde\varphi\left(\frac{|y-x|}{2}\right)\right)dr\\
    &=2\int_{\R}\tilde\rho(\xi)\left(\tilde\varphi\left(\frac{|y-x|}{2}-\xi\right)-\tilde\varphi\left(\frac{|y-x|}{2}\right)\right)d\xi\\
    &=2\int_{\R}\tilde\rho(\xi)\left(\tilde\varphi\left(\frac{|y-x|}{2}+\xi\right)-\tilde\varphi\left(\frac{|y-x|}{2}\right)\right)d\xi\\
    &\leq 2 \tilde\varphi_t(s).
\end{align*}
For the second term $I_\Psi$, use the symmetry $\rho(r)=\rho(Rr)$ and apply the change of variables $r'=Rr$ to the first integral in the second line in the calculation below:
\begin{align*}
    I_\Psi&=E\int_{\R^n}\rho(r)\left(\Psi(y+Rr)-\Psi(y) +\Psi(x+r)-\Psi(x)\right)dr\\
    &=E\int_{\R^n} \rho(r) (\Psi(y+Rr)-\Psi(y))dr+E\int_{\R^n}\rho(r) (\Psi(x+r)-\Psi(x))dr\\
    &=E(L_{\R^n}\Psi(y)+L_{\R^n}\Psi(x)).
\end{align*}
By condition (d.), the last two terms are bounded above by $K$, so
\begin{align*}
    I_\Psi\leq 2EK.
\end{align*}

\textbf{Step 5: The contradiction.} Substituting the results from Step 4 into the final inequality from Step 3 gives
\begin{align*}
    0&\leq I_\varphi+I_\Psi-2\varphi_t(s)-CE(\Psi(x)+\Psi(y))\\
    &\leq 2\tilde\varphi_t(s)+2EK-2\varphi_t(s)-CE(\Psi(x)+\Psi(y))\\
    &=2EK-CE(\Psi(x)+\Psi(y)),
\end{align*}
where $\tilde\varphi_t(s)=\varphi_t(s)$ because $s>0$.
Since $\Psi(x) \ge 1$ and $\Psi(y) \ge 1$, we have $\Psi(x)+\Psi(y) \ge 2$. Thus,
\begin{align*}
0 &\le 2EK - 2CE \\
0 &\le 2E(K-C).
\end{align*}
As $E>0$ and we chose $C > K$, it follows that $K-C < 0$. This yields the contradiction.

Hence, $Z_{\epsilon}<0$ for each $\epsilon>0$. Letting $\epsilon \to 0$ completes the proof.
\end{proof}

\begin{rem}[One-dimensional solutions and sharpness]\label{rem:one_dimensional_sharpness}
Condition~\textup{(c.)} is naturally related to solutions of the $n$-dimensional equation that depend on only one spatial variable: suppose that $\tilde\varphi$ satisfies the one-dimensional equation in condition~\textup{(c.)} with equality, and define
\[
U(x,t):=\tilde\varphi(x_1,t),
\qquad x=(x_1,p)\in\R\times\R^{n-1}.
\]
Using the definition of the reduced kernel, we have
\begin{align*}
L_{\R^n}U(x,t)
&=\int_{\R}\int_{\R^{n-1}}\rho(w,q)
\bigl(\tilde\varphi(x_1+w,t)-\tilde\varphi(x_1,t)\bigr)\,dq\,dw\\
&=\int_{\R}\tilde\rho(w)
\bigl(\tilde\varphi(x_1+w,t)-\tilde\varphi(x_1,t)\bigr)\,dw\\
&=\tilde\varphi_t(x_1,t)=U_t(x,t).
\end{align*}
Thus the one-dimensional equation in condition~\textup{(c.)} is the equation obtained by restricting the $n$-dimensional evolution to this class of solutions.

If we additionally assume that $\varphi(\cdot,t)$ is concave for each
$t>0$, then $\varphi(\cdot,t)$ is a modulus of continuity for
$U(\cdot,t)$. To see this, fix $t>0$. Since $\varphi(0,t)=0$ and
$\varphi(\cdot,t)$ is concave, $\varphi(\cdot,t)$ is subadditive.
Let
\[
    x=(x_1,x'),\qquad y=(y_1,y')\in\R^n.
\]
Suppose first that $x_1$ and $y_1$ have the same sign. By exchanging
$x$ and $y$ if necessary, and using the oddness of $\tilde\varphi$ to
reflect the pair when both first coordinates are negative, we may
assume without loss of generality that $0\leq x_1\leq y_1$. Since
$y_1=x_1+(y_1-x_1)$, subadditivity gives
\[
    \varphi(y_1,t)
    \leq
    \varphi(x_1,t)+\varphi(y_1-x_1,t),
\]
and hence
\[
    |\tilde\varphi(y_1,t)-\tilde\varphi(x_1,t)|
    \leq
    \varphi(y_1-x_1,t).
\]
Concavity and $\varphi(0,t)=0$ imply
\[
    \varphi\!\left(\frac{y_1-x_1}{2},t\right)
    \geq
    \frac12\varphi(y_1-x_1,t),
\]
so
\[
    |\tilde\varphi(y_1,t)-\tilde\varphi(x_1,t)|
    \leq
    2\varphi\!\left(\frac{y_1-x_1}{2},t\right).
\]

Now suppose that $x_1$ and $y_1$ have opposite signs. By
exchanging $x$ and $y$ if necessary, we may assume that $x_1\leq0\leq y_1$. Using the oddness of $\tilde\varphi$ and mid-point concavity,
\[
    |\tilde\varphi(y_1,t)-\tilde\varphi(x_1,t)|
    =
    \varphi(y_1,t)+\varphi(-x_1,t)\leq
    2\varphi\!\left(\frac{y_1-x_1}{2},t\right).
\]

In either case,
\[
    |U(y,t)-U(x,t)|
    \leq
    2\varphi\!\left(\frac{|y_1-x_1|}{2},t\right).
\]
Since $|y_1-x_1|\leq |y-x|$ and $\varphi(\cdot,t)$ is increasing, it follows that
\[
    |U(y,t)-U(x,t)|
    \leq
    2\varphi\!\left(\frac{|y-x|}{2},t\right).
\]
Hence $\varphi(\cdot,t)$ is a modulus of continuity for
$U(\cdot,t)$.

For the symmetric pair
\[
    x=-se_1,\qquad y=se_1,
\]
we have $|y-x|=2s$, and therefore
\[
    |U(y,t)-U(x,t)|
    =
    \tilde\varphi(s,t)-\tilde\varphi(-s,t)
    =
    2\varphi(s,t)
    =
    2\varphi\!\left(\frac{|y-x|}{2},t\right).
\]
The modulus estimate is consequently attained by symmetric pairs of points for these one-dimensional solutions. Thus the bound in Theorem~\ref{thm_modulus_Rn} is sharp as a uniform
modulus-of-continuity estimate.
\end{rem}

\subsection{Singular kernels}
We now explain how Theorem~\ref{thm_modulus_Rn} extends to
principal-value kernels satisfying the Lévy integrability condition
\eqref{eq:levy_condition}, including the fractional Laplacian.

\medskip

The key point is that for singular kernels, the operator is no longer given
by an absolutely convergent integral, but must be interpreted in the
principal-value sense. Thus, the direct use of the full-space integrals in
Steps~2 and~3 of the proof of Theorem~\ref{thm_modulus_Rn} must be replaced
by truncated integrals and a limit procedure. We cannot directly
insert the truncation $\mathbf{1}_{\R^n\setminus B_\delta(0)}$ into every
integral in Step~3, however, and reuse the proof verbatim, because the change of
variables $g$ used to pair the $L$- and $R$-regions does \emph{not} preserve
the ball $B_\delta(0)$. This produces an additional small-ball remainder
term, which must be estimated separately. That is the only new issue.

\begin{lem}[Truncated reflection identity]
\label{lem:truncated_reflection}
Let \(K:\R^n\setminus\{0\}\to[0,\infty)\) be radial and non-increasing in
\(|r|\), and assume that the truncated integrals below are finite. With the
notation of Step~3 in the proof of Theorem~\ref{thm_modulus_Rn}, let
\[
L_K^\delta f(x):=\int_{|r|>\delta}K(r)\bigl(f(x+r)-f(x)\bigr)\,dr,
\qquad \delta>0.
\]
Then for every \(0<\delta<s\),
\begin{align*}
L_K^\delta u(y)-L_K^\delta u(x)
&=
\int_{L\setminus B_\delta(0)}\bigl(K(r)-K(g(r))\bigr)
\bigl(u(y+Rr)-u(x+r)\bigr)\,dr \\
&\qquad
-\int_{|r|>\delta}K(r)\bigl(u(y)-u(x)\bigr)\,dr
+\mathcal R_\delta,
\end{align*}
where
\[
\mathcal R_\delta:=
-\int_{L\cap B_\delta(0)}K(g(r))
\bigl(u(y+Rr)-u(x+r)\bigr)\,dr.
\]
Moreover, if \(u\) is bounded, then
\[
\mathcal R_\delta\to0
\qquad\text{as }\delta\downarrow0.
\]
\end{lem}

\begin{proof}
We recall the notation from Step~3 of the proof of
Theorem~\ref{thm_modulus_Rn}. Let
\[
s:=\frac{|x-y|}{2},
\qquad
e_1:=\frac{y-x}{|y-x|}.
\]
Thus \(y-x=2s e_1\). Let \(R\) denote reflection in the hyperplane
orthogonal to \(e_1\), namely
\[
Rr:=r-2\langle r,e_1\rangle e_1.
\]
Then \(R\) is orthogonal, so
\[
|Rr|=|r|,
\qquad
|\det R|=1.
\]
Since \(K\) is radial, \(K(Rr)=K(r)\).

We also recall the pairing map
\[
g(r):=(y-x)+Rr=2s e_1+Rr.
\]
The three regions are
\[
L:=\{r:\langle r,e_1\rangle<s\},
\qquad
S:=\{r:\langle r,e_1\rangle=s\},
\qquad
R:=\{r:\langle r,e_1\rangle>s\}.
\]
The map \(g\) swaps \(L\) and \(R\). Indeed,
\[
\langle g(r),e_1\rangle
=
2s-\langle r,e_1\rangle.
\]
Thus \(r\in L\), i.e. \(\langle r,e_1\rangle<s\), implies
\(\langle g(r),e_1\rangle>s\), so \(g(r)\in R\). Similarly, \(r\in R\)
implies \(g(r)\in L\). Moreover,
\[
g(g(r))=r,
\]
since \(R^2=\mathrm{id}\) and \(R(y-x)=-(y-x)\). Hence \(g\) is a
bijection from \(L\) onto \(R\). Since \(g\) is a translation composed with
an orthogonal reflection, it preserves Lebesgue measure.

We shall use the identities
\[
x+g(r)=y+Rr,
\qquad
y+Rg(r)=x+r.
\]

Set
\[
F(r):=u(y+Rr)-u(x+r).
\]
By definition,
\begin{align*}
L_K^\delta u(y)-L_K^\delta u(x)
&=
\int_{|r|>\delta}K(r)\bigl(u(y+r)-u(y)\bigr)\,dr \\
&\qquad
-\int_{|r|>\delta}K(r)\bigl(u(x+r)-u(x)\bigr)\,dr.
\end{align*}
In the first integral, change variables by the reflection \(r\mapsto Rr\).
Since \(R\) is orthogonal, it preserves Lebesgue measure and satisfies
\(|Rr|=|r|\). As \(K\) is radial, \(K(Rr)=K(r)\). Therefore
\[
\int_{|r|>\delta}K(r)u(y+r)\,dr
=
\int_{|r|>\delta}K(r)u(y+Rr)\,dr.
\]
Here the truncated region \(\{|r|>\delta\}\) is invariant under the
reflection because \(|Rr|=|r|\). Therefore
\begin{align}
L_K^\delta u(y)-L_K^\delta u(x)
&=
\int_{|r|>\delta}K(r)\bigl(u(y+Rr)-u(x+r)\bigr)\,dr \notag\\
&\qquad
-\int_{|r|>\delta}K(r)\bigl(u(y)-u(x)\bigr)\,dr \notag\\
&=
\int_{|r|>\delta}K(r)F(r)\,dr
-\int_{|r|>\delta}K(r)\bigl(u(y)-u(x)\bigr)\,dr.
\label{eq:truncated-reflection-start}
\end{align}

We now decompose the first integral over \(L\), \(S\), and \(R\). Since
\(0<\delta<s\), the sets \(S\) and \(R\) do not meet \(B_\delta(0)\).
Indeed, if \(r\in S\), then
\[
|r|\ge |\langle r,e_1\rangle|=s>\delta,
\]
whereas if \(r\in R\), then
\[
|r|\ge \langle r,e_1\rangle>s>\delta.
\]
Thus
\[
S\cap B_\delta(0)=\varnothing,
\qquad
R\cap B_\delta(0)=\varnothing.
\]
Consequently,
\[
\{|r|>\delta\}
=
\bigl(L\setminus B_\delta(0)\bigr)\cup S\cup R,
\]
up to disjoint union, and hence
\begin{align}
\int_{|r|>\delta}K(r)F(r)\,dr
&=
\int_{L\setminus B_\delta(0)}K(r)F(r)\,dr
+\int_S K(r)F(r)\,dr
+\int_R K(r)F(r)\,dr.
\label{eq:truncated-LSR-split}
\end{align}

If \(r\in S\), then \(\langle r,e_1\rangle=s\), and so
\[
Rr=r-2s e_1=r-(y-x).
\]
Hence
\[
y+Rr=y+r-(y-x)=x+r.
\]
Therefore \(F(r)=0\) on \(S\), and so
\[
\int_S K(r)F(r)\,dr=0.
\]

It remains to rewrite the integral over \(R\). Since \(g:L\to R\) is a
measure-preserving bijection, the change of variables \(r'=g(r)\) gives
\[
\int_R K(r')F(r')\,dr'
=
\int_L K(g(r))F(g(r))\,dr.
\]
Using
\[
x+g(r)=y+Rr,
\qquad
y+Rg(r)=x+r,
\]
we have
\begin{align*}
F(g(r))
&=
u(y+Rg(r))-u(x+g(r)) \\
&=
u(x+r)-u(y+Rr) \\
&=
-F(r).
\end{align*}
Thus
\[
\int_R K(r)F(r)\,dr
=
-\int_L K(g(r))F(r)\,dr.
\]
Substituting this and the vanishing of the \(S\)-term into
\eqref{eq:truncated-LSR-split}, we get
\[
\int_{|r|>\delta}K(r)F(r)\,dr
=
\int_{L\setminus B_\delta(0)}K(r)F(r)\,dr
-
\int_L K(g(r))F(r)\,dr.
\]
Splitting the second integral over
\[
L=\bigl(L\setminus B_\delta(0)\bigr)\cup\bigl(L\cap B_\delta(0)\bigr),
\]
we obtain
\begin{align*}
-\int_L K(g(r))F(r)\,dr
&=
-\int_{L\setminus B_\delta(0)}K(g(r))F(r)\,dr
-\int_{L\cap B_\delta(0)}K(g(r))F(r)\,dr.
\end{align*}
Therefore
\begin{align*}
\int_{|r|>\delta}K(r)F(r)\,dr
&=
\int_{L\setminus B_\delta(0)}
\bigl(K(r)-K(g(r))\bigr)F(r)\,dr \\
&\qquad
-\int_{L\cap B_\delta(0)}K(g(r))F(r)\,dr.
\end{align*}
Define
\[
\mathcal R_\delta
:=
-\int_{L\cap B_\delta(0)}K(g(r))F(r)\,dr.
\]
Since \(F(r)=u(y+Rr)-u(x+r)\), this is exactly
\[
\mathcal R_\delta
=
-\int_{L\cap B_\delta(0)}
K(g(r))\bigl(u(y+Rr)-u(x+r)\bigr)\,dr.
\]
Substituting into \eqref{eq:truncated-reflection-start} gives
\begin{align*}
L_K^\delta u(y)-L_K^\delta u(x)
&=
\int_{L\setminus B_\delta(0)}
\bigl(K(r)-K(g(r))\bigr)
\bigl(u(y+Rr)-u(x+r)\bigr)\,dr \\
&\qquad
-\int_{|r|>\delta}K(r)\bigl(u(y)-u(x)\bigr)\,dr
+\mathcal R_\delta.
\end{align*}
This proves the identity.

It remains to show that \(\mathcal R_\delta\to0\) as
\(\delta\downarrow0\) when \(u\) is bounded. Suppose
\[
\|u\|_{L^\infty}<\infty.
\]
If \(r\in L\cap B_\delta(0)\), then \(|r|<\delta<s\). Since \(R\) is an
isometry and \(g(r)=(y-x)+Rr\), the reverse triangle inequality gives
\[
|g(r)|
\ge |y-x|-|Rr|
=
2s-|r|
>
2s-s=s.
\]
Thus \(g(r)\) remains away from the origin. Since \(K\) is radial and
non-increasing in \(|r|\),
\[
K(g(r))\le K(se_1)<\infty.
\]
Therefore
\begin{align*}
|\mathcal R_\delta|
&\le
\int_{L\cap B_\delta(0)}
K(g(r))
\bigl|u(y+Rr)-u(x+r)\bigr|\,dr \\
&\le
2\|u\|_{L^\infty}
\int_{L\cap B_\delta(0)}K(g(r))\,dr \\
&\le
2\|u\|_{L^\infty}K(se_1)\,|L\cap B_\delta(0)| \\
&\le
2\|u\|_{L^\infty}K(se_1)\,|B_\delta(0)|.
\end{align*}
Since \(|B_\delta(0)|\to0\) as \(\delta\downarrow0\), we conclude that
\[
\mathcal R_\delta\to0.
\]
The proof is complete.
\end{proof}

We now apply Lemma~\ref{lem:truncated_reflection} to prove an extension of
Theorem~\ref{thm_modulus_Rn} to principal-value kernels satisfying the Lévy
condition.

\begin{thm}[Modulus of continuity for principal-value kernels on $\R^n$]
\label{thm:pv_modulus_continuity}
Let $K:\R^n\setminus\{0\}\to[0,\infty)$ be measurable, even, radial, and
non-increasing in $|r|$, and assume
\begin{equation}\label{eq:levy_condition}
\int_{\R^n}(1\wedge |r|^2)\,K(r)\,dr<\infty.
\end{equation}
For any function \(f\) for which the following principal-value limit exists,
write
\begin{equation}\label{eq:LK_definition}
L_K f(x):=\PV\int_{\R^n}K(r)\bigl(f(x+r)-f(x)\bigr)\,dr
:=\lim_{\delta\downarrow0}\int_{|r|>\delta}K(r)\bigl(f(x+r)-f(x)\bigr)\,dr.
\end{equation}

Let $u:\R^n\times[0,\infty)\to\R$ satisfy:
\begin{enumerate}
\item[(i)] $u$ is bounded, with
\[
\|u\|_{L^\infty(\R^n\times[0,\infty))}\le 1;
\]
\item[(ii)] for each $t>0$, $u(\cdot,t)\in C^{1,1}_{\mathrm{loc}}(\R^n)$; and
\item[(iii)] the equation
\[
u_t(x,t)=L_Ku(\cdot,t)(x)
\qquad\text{for all }(x,t)\in\R^n\times(0,\infty),
\]
holds in the pointwise classical sense.
\end{enumerate}

Let $\varphi:[0,\infty)\times[0,\infty)\to[0,\infty)$ satisfy:
\begin{enumerate}
\item[(a)] for each $t\ge0$, the map $r\mapsto\varphi(r,t)$ belongs to
\(C^1([0,\infty))\), is strictly increasing, and satisfies \(\varphi(0,t)=0\);
\item[(b)] $\varphi(\cdot,0)$ is a modulus of continuity for $u(\cdot,0)$; and
\item[(c)] for each $t>0$, the odd extension
\[
\tilde\varphi(r,t):=
\begin{cases}
\varphi(r,t), & r\ge0,\\
-\varphi(-r,t), & r<0,
\end{cases}
\]
belongs to $C^{1,1}_{\mathrm{loc}}(\R)$, satisfies
\[
\lim_{r\to\pm\infty}\tilde\varphi(r,t)=\pm1,
\]
and is a supersolution of
\begin{equation}\label{eq:1d_pv}
\tilde\varphi_t(s,t)
\geq
\PV\int_{\R}\tilde K(w)\bigl(\tilde\varphi(s+w,t)-\tilde\varphi(s,t)\bigr)\,dw, \qquad s\in\R^+, t>0
\end{equation}
where
\[
\tilde K:\R\setminus\{0\}\to[0,\infty),
\qquad
\tilde K(w):=\int_{\R^{n-1}}K(w,p)\,dp
\quad (w\neq0).
\]
\end{enumerate}

Assume further that there exists a smooth radial function
\(\Psi:\R^n\to[1,\infty)\) such that
\[
\Psi(x)\to\infty\quad\text{as }|x|\to\infty,
\]
and such that \(L_K\Psi(x)\), defined by \eqref{eq:LK_definition}, exists
for every \(x\in\R^n\) and satisfies
\[
\sup_{x\in\R^n}L_K\Psi(x)\le K_0<\infty.
\]

Then, for every $t\ge0$, the function $\varphi(\cdot,t)$ is a modulus of
continuity for $u(\cdot,t)$:
\[
|u(y,t)-u(x,t)|\le 2\varphi\!\left(\frac{|x-y|}{2},t\right)
\qquad\text{for all }x,y\in\R^n.
\]
\end{thm}
\begin{proof}
We follow the proof of Theorem~\ref{thm_modulus_Rn}, indicating the points
where the principal-value setting requires a modification.

Fix \(\varepsilon>0\) and choose \(C>K_0\). Define
\[
Z_\varepsilon(x,y,t)
:=
u(y,t)-u(x,t)-2\varphi\!\left(\frac{|x-y|}{2},t\right)
-\varepsilon e^{Ct}\bigl(\Psi(x)+\Psi(y)\bigr).
\]
As in Step~1 of the proof of Theorem~\ref{thm_modulus_Rn}, it suffices to
prove that
\[
Z_\varepsilon(x,y,t)<0
\qquad\text{for all }x,y\in\R^n,\ t\ge0.
\]
Letting \(\varepsilon\downarrow0\) will then give the desired modulus
estimate.

The cases \(x=y\) and \(t=0\) are handled exactly as in
Theorem~\ref{thm_modulus_Rn}, using \(\varphi(0,t)=0\) and the fact that
\(\varphi(\cdot,0)\) is a modulus of continuity for \(u(\cdot,0)\). Since
\(u\) is bounded, \(\varphi\ge0\), and \(\Psi(z)\to\infty\) as
\(|z|\to\infty\), we also have
\[
Z_\varepsilon(x,y,t)\to-\infty
\qquad\text{as }\max\{|x|,|y|\}\to\infty,
\]
for each fixed \(t>0\).

Thus the same first-contact argument as in Step~2 of
Theorem~\ref{thm_modulus_Rn} applies. If \(Z_\varepsilon\) is nonnegative
somewhere, then there exists a first contact point
\[
(x,y,t)\in\R^n\times\R^n\times(0,\infty)
\]
with \(x\ne y\), such that
\[
Z_\varepsilon(x,y,t)=0,
\qquad
Z_\varepsilon(a,b,\tau)\le0
\quad\text{for all }a,b\in\R^n,\ 0\le\tau\le t.
\]
At this first contact point,
\[
0\le \partial_t Z_\varepsilon(x,y,t).
\]

Set
\[
s:=\frac{|x-y|}{2},
\qquad
E:=\varepsilon e^{Ct}.
\]
For \(\delta>0\), write
\[
L_K^\delta f(x):=
\int_{|r|>\delta}K(r)\bigl(f(x+r)-f(x)\bigr)\,dr.
\]
Since \(u_t=L_Ku\) pointwise in the principal-value sense,
\[
0\le
\partial_t Z_\varepsilon(x,y,t)
=
\lim_{\delta\downarrow0}
\Bigl(
L_K^\delta u(y)-L_K^\delta u(x)
-2\varphi_t(s,t)
-CE(\Psi(x)+\Psi(y))
\Bigr).
\]
Define
\[
A_\delta
:=
L_K^\delta u(y)-L_K^\delta u(x)
-2\varphi_t(s,t)
-CE(\Psi(x)+\Psi(y)).
\]
Then
\[
0\le \lim_{\delta\downarrow0}A_\delta.
\]

We now use the reflection notation from Step~3 of
Theorem~\ref{thm_modulus_Rn}. Set
\[
e_1:=\frac{y-x}{|x-y|},
\qquad
Rr:=r-2\langle r,e_1\rangle e_1,
\qquad
g(r):=(y-x)+Rr,
\qquad
r_1:=\langle r,e_1\rangle.
\]
Also define
\[
L:=\{r\in\R^n:\langle r,e_1\rangle<s\},
\qquad
S:=\{r\in\R^n:\langle r,e_1\rangle=s\},
\qquad
R:=\{r\in\R^n:\langle r,e_1\rangle>s\}.
\]

For \(0<\delta<s\), Lemma~\ref{lem:truncated_reflection} gives
\begin{align*}
A_\delta
&=
\int_{L\setminus B_\delta(0)}
\bigl(K(r)-K(g(r))\bigr)
\bigl(u(y+Rr)-u(x+r)\bigr)\,dr \\
&\qquad
-\int_{|r|>\delta}K(r)\bigl(u(y)-u(x)\bigr)\,dr \\
&\qquad
-2\varphi_t(s,t)
-CE(\Psi(x)+\Psi(y))
+\mathcal R_\delta.
\end{align*}

At the first contact point, for every \(r\), we have
\[
Z_\varepsilon(x+r,y+Rr,t)\le0.
\]
Hence
\[
u(y+Rr)-u(x+r)
\le
2\varphi\!\left(\frac{|(x+r)-(y+Rr)|}{2},t\right)
+
E\bigl(\Psi(x+r)+\Psi(y+Rr)\bigr).
\]
For \(r\in L\),
\[
\frac{|(x+r)-(y+Rr)|}{2}=s-r_1.
\]
Moreover, \(K(r)-K(g(r))\ge0\) on \(L\), by the radial monotonicity of
\(K\) and the reflection geometry. Therefore
\begin{align*}
A_\delta
&\le
2\int_{L\setminus B_\delta(0)}
\bigl(K(r)-K(g(r))\bigr)\varphi(s-r_1,t)\,dr \\
&\qquad
+E\int_{L\setminus B_\delta(0)}
\bigl(K(r)-K(g(r))\bigr)
\bigl(\Psi(x+r)+\Psi(y+Rr)\bigr)\,dr \\
&\qquad
-\int_{|r|>\delta}K(r)\bigl(u(y)-u(x)\bigr)\,dr \\
&\qquad
-2\varphi_t(s,t)
-CE(\Psi(x)+\Psi(y))
+\mathcal R_\delta.
\end{align*}

Since \(Z_\varepsilon(x,y,t)=0\), we have
\[
u(y,t)-u(x,t)
=
2\varphi(s,t)+E(\Psi(x)+\Psi(y)).
\]
Substituting this identity into the preceding inequality gives
\begin{align*}
A_\delta
&\le
\Bigg[
2\int_{L\setminus B_\delta(0)}
\bigl(K(r)-K(g(r))\bigr)\varphi(s-r_1,t)\,dr
-
2\varphi(s,t)\int_{|r|>\delta}K(r)\,dr
\Bigg] \\
&\qquad
+
E\Bigg[
\int_{L\setminus B_\delta(0)}
\bigl(K(r)-K(g(r))\bigr)
\bigl(\Psi(x+r)+\Psi(y+Rr)\bigr)\,dr \\
&\hspace{5.3cm}
-
(\Psi(x)+\Psi(y))\int_{|r|>\delta}K(r)\,dr
\Bigg] \\
&\qquad
-2\varphi_t(s,t)
-CE(\Psi(x)+\Psi(y))
+\mathcal R_\delta.
\end{align*}

Define
\[
I_{\varphi,\delta}
:=
2\int_{|r|>\delta}K(r)
\bigl(\tilde\varphi(s-r_1,t)-\tilde\varphi(s,t)\bigr)\,dr
\]
and
\[
I_{\Psi,\delta}
:=
E\int_{|r|>\delta}K(r)
\bigl(\Psi(y+Rr)-\Psi(y)+\Psi(x+r)-\Psi(x)\bigr)\,dr.
\]

We first rewrite the expression involving \(\varphi\), namely
\[
2\int_{L\setminus B_\delta(0)}
\bigl(K(r)-K(g(r))\bigr)\varphi(s-r_1,t)\,dr
-
2\varphi(s,t)\int_{|r|>\delta}K(r)\,dr.
\]
This expression is not exactly \(I_{\varphi,\delta}\), because the
truncation \(\{|r|>\delta\}\) is not invariant under the map \(g\).
Decomposing the integral defining \(I_{\varphi,\delta}\) over
\(L\), \(S\), and \(R\); using that the \(S\)-contribution has measure zero;
and changing variables on \(R\) by \(g:L\to R\), gives
\begin{align*}
\frac12 I_{\varphi,\delta}
&=
\int_{L\setminus B_\delta(0)}
K(r)\bigl(\varphi(s-r_1,t)-\varphi(s,t)\bigr)\,dr \\
&\qquad
-\int_L K(g(r))
\bigl(\varphi(s-r_1,t)+\varphi(s,t)\bigr)\,dr.
\end{align*}
Also,
\[
\int_{|r|>\delta}K(r)\,dr
=
\int_{L\setminus B_\delta(0)}K(r)\,dr
+
\int_L K(g(r))\,dr,
\]
again by the \(L/S/R\) decomposition and the change of variables
\(g:L\to R\). It follows that
\begin{align*}
&2\int_{L\setminus B_\delta(0)}
\bigl(K(r)-K(g(r))\bigr)\varphi(s-r_1,t)\,dr
-
2\varphi(s,t)\int_{|r|>\delta}K(r)\,dr \\
&\qquad
=
I_{\varphi,\delta}
+
J_{\varphi,\delta},
\end{align*}
where
\[
J_{\varphi,\delta}
:=
2\int_{L\cap B_\delta(0)}
K(g(r))\varphi(s-r_1,t)\,dr.
\]

Next consider the expression involving \(\Psi\). Since
\[
K(r)-K(g(r))\ge0
\qquad\text{on }L
\]
and \(\Psi\ge1\), we have
\begin{align*}
& E\Bigg[
\int_{L\setminus B_\delta(0)}
\bigl(K(r)-K(g(r))\bigr)
\bigl(\Psi(x+r)+\Psi(y+Rr)\bigr)\,dr \\
&\hspace{5.2cm}
-
(\Psi(x)+\Psi(y))\int_{|r|>\delta}K(r)\,dr
\Bigg] \\
&\qquad
\le
E\int_{|r|>\delta}K(r)
\bigl(\Psi(x+r)+\Psi(y+Rr)\bigr)\,dr
-
E(\Psi(x)+\Psi(y))\int_{|r|>\delta}K(r)\,dr \\
&\qquad
=
I_{\Psi,\delta}.
\end{align*}
Consequently,
\[
A_\delta
\le
I_{\varphi,\delta}
+
I_{\Psi,\delta}
-2\varphi_t(s,t)
-CE(\Psi(x)+\Psi(y))
+\mathcal R_\delta
+J_{\varphi,\delta}.
\]

We now show that
\[
J_{\varphi,\delta}\to0
\qquad\text{as }\delta\downarrow0.
\]
If \(r\in L\cap B_\delta(0)\) and \(0<\delta<s\), then the same estimate as
in Lemma~\ref{lem:truncated_reflection} gives
\[
|g(r)|\ge |y-x|-|r|>2s-s=s.
\]
Hence, since \(K\) is radial and non-increasing in \(|r|\),
\[
K(g(r))\le K(se_1)<\infty.
\]
Taking \(0<\delta<s/2\), we also have
\[
|r_1|\le |r|<\delta<s/2,
\]
and therefore
\[
s-r_1\in(s-\delta,s+\delta)\subset(s/2,3s/2).
\]
Since \(\varphi(\cdot,t)\) is continuous, it is bounded on
\([s/2,3s/2]\). Thus there is a constant \(C_{\varphi,s,t}\) such that
\[
0\le J_{\varphi,\delta}
\le
C_{\varphi,s,t}K(se_1)|L\cap B_\delta(0)|
\le
C_{\varphi,s,t}K(se_1)|B_\delta(0)|.
\]
Since \(|B_\delta(0)|\to0\), we conclude that
\[
J_{\varphi,\delta}\to0.
\]

It remains to pass to the limit in \(I_{\varphi,\delta}\) and
\(I_{\Psi,\delta}\).

For the \(\varphi\)-term, set
\[
d(w):=\tilde\varphi(s-w,t)-\tilde\varphi(s,t).
\]
Thus
\[
I_{\varphi,\delta}
=
2\int_{|r|>\delta}K(r)d(r_1)\,dr.
\]
Near the origin, subtract the linear part by defining
\[
\eta(w):=
d(w)+\tilde\varphi_r(s,t)w
=
\tilde\varphi(s-w,t)-\tilde\varphi(s,t)
+\tilde\varphi_r(s,t)w.
\]
Since \(\tilde\varphi(\cdot,t)\in C^{1,1}_{\mathrm{loc}}(\R)\), the
first-order Taylor remainder at \(s\) is quadratic. Hence there exist
\(\rho\in(0,1)\) and \(c>0\) such that
\[
|\eta(w)|\le c|w|^2
\qquad\text{for }|w|<\rho.
\]

For \(0<\delta<\rho\), split
\[
I_{\varphi,\delta}
=
2\int_{\delta<|r|<\rho}K(r)d(r_1)\,dr
+
2\int_{|r|\ge\rho}K(r)d(r_1)\,dr.
\]
On the annulus \(\{\delta<|r|<\rho\}\),
\[
d(r_1)=\eta(r_1)-\tilde\varphi_r(s,t)r_1.
\]
The linear term integrates to zero on this symmetric annulus:
\[
\int_{\delta<|r|<\rho}K(r)\,\tilde\varphi_r(s,t)r_1\,dr=0,
\]
because \(K\) is even, \(r_1=\langle r,e_1\rangle\) is odd in \(r\), and
the annulus is invariant under \(r\mapsto -r\). Therefore
\[
\int_{\delta<|r|<\rho}K(r)d(r_1)\,dr
=
\int_{\delta<|r|<\rho}K(r)\eta(r_1)\,dr.
\]
Since
\[
|\eta(r_1)|\le c|r_1|^2\le c|r|^2,
\]
the near-field integrand is dominated by \(cK(r)|r|^2\), which is
integrable near the origin by the Lévy condition. Hence dominated
convergence gives
\[
\lim_{\delta\downarrow0}
\int_{\delta<|r|<\rho}K(r)d(r_1)\,dr
=
\int_{|r|<\rho}K(r)\eta(r_1)\,dr.
\]

For the far-field part, \(d\) is bounded, since
\(\tilde\varphi(\cdot,t)\) is bounded on \(\R\). Moreover,
\[
\int_{|r|\ge\rho}K(r)\,dr<\infty.
\]
Indeed, on \(\rho\le |r|<1\),
\[
K(r)\le \rho^{-2}|r|^2K(r),
\]
and \(|r|^2K(r)\) is integrable near the origin by the Lévy condition; while
on \(|r|\ge1\), the Lévy condition gives
\[
\int_{|r|\ge1}K(r)\,dr<\infty.
\]
Thus the far-field integral is absolutely convergent.

Combining the near-field and far-field parts gives
\[
\lim_{\delta\downarrow0}I_{\varphi,\delta}
=
2\,\PV\int_{\R^n}K(r)
\bigl(\tilde\varphi(s-r_1,t)-\tilde\varphi(s,t)\bigr)\,dr.
\]
The singular linear part is odd and cancels under symmetric truncations,
while the quadratic remainder near the origin and the bounded tail are
absolutely integrable by the estimates above. Hence Fubini's theorem gives
\[
\PV\int_{\R^n}K(r)
\bigl(\tilde\varphi(s-r_1,t)-\tilde\varphi(s,t)\bigr)\,dr
=
\PV\int_{\R}\tilde K(w)
\bigl(\tilde\varphi(s-w,t)-\tilde\varphi(s,t)\bigr)\,dw.
\]
Since \(\tilde K\) is even, the change of variables \(w\mapsto -w\) gives
\[
\PV\int_{\R}\tilde K(w)
\bigl(\tilde\varphi(s-w,t)-\tilde\varphi(s,t)\bigr)\,dw
=
\PV\int_{\R}\tilde K(w)
\bigl(\tilde\varphi(s+w,t)-\tilde\varphi(s,t)\bigr)\,dw.
\]
Therefore, by the one-dimensional supersolution inequality satisfied by
\(\tilde\varphi\),
\[
\lim_{\delta\downarrow0}I_{\varphi,\delta}
\leq 2\tilde\varphi_t(s,t).
\]

For the \(\Psi\)-term, using the change of variables \(r\mapsto Rr\) in
the first term, we have
\[
I_{\Psi,\delta}
=
E\bigl(L_K^\delta\Psi(y)+L_K^\delta\Psi(x)\bigr).
\]
By the assumption on \(\Psi\),
\[
\lim_{\delta\downarrow0}I_{\Psi,\delta}
=
E\bigl(L_K\Psi(y)+L_K\Psi(x)\bigr)
\le 2EK_0.
\]

Taking \(\limsup_{\delta\downarrow0}\) in the inequality for \(A_\delta\),
and using
\[
0\le \lim_{\delta\downarrow0}A_\delta,
\qquad
\mathcal R_\delta\to0,
\qquad
J_{\varphi,\delta}\to0,
\]
together with
\[
\lim_{\delta\downarrow0}I_{\varphi,\delta}\leq 2\tilde\varphi_t(s,t),
\qquad
\lim_{\delta\downarrow0}I_{\Psi,\delta}\le2EK_0,
\]
we obtain
\[
0
\le
2\tilde\varphi_t(s,t)
+
2EK_0
-
2\varphi_t(s,t)
-
CE(\Psi(x)+\Psi(y)).
\]
Since \(s>0\), we have
\[
\tilde\varphi_t(s,t)=\varphi_t(s,t).
\]
Therefore
\[
0\le 2EK_0-CE(\Psi(x)+\Psi(y)).
\]
Finally, since \(\Psi\ge1\),
\[
0\le 2EK_0-2CE=2E(K_0-C),
\]
which is impossible because \(E>0\) and \(C>K_0\).

This contradiction shows that
\[
Z_\varepsilon(x,y,t)<0
\qquad\text{for all }x,y\in\R^n,\ t\ge0.
\]
Letting \(\varepsilon\downarrow0\) completes the proof.
\end{proof}

\begin{rem}
The fractional Laplacian corresponds to
\[
K(r)=C_{n,s}|r|^{-n-2s},
\qquad s\in(0,1),
\]
in which case
\[
L_K=-(-\Delta)^s
\]
and
\[
\tilde K(w)=C_{1,s}|w|^{-1-2s}.
\]
Accordingly, the one-dimensional supersolution condition \eqref{eq:1d_pv} becomes
\[
\tilde{\varphi}_t(r,t)
\geq
-(-\Delta)^s \tilde{\varphi}(r,t),
\]
with equality for solutions of the one-dimensional fractional heat equation.

The regularisation assumptions are satisfied by choosing
\[
\Psi(x)=(1+|x|^2)^{\alpha/2}
\qquad\text{for any }0<\alpha<2s.
\]
See Appendix~\ref{app:barrier}.
\end{rem}

\section{Regional heat equation}\label{sec_regional_heat_equation}
Now we consider the non-local heat equation on a bounded convex set $\Omega$ in $\R^n$. This gives rise to the `regional' heat equation, where the non-local interactions are restricted to the domain $\Omega$. We define the regional operator $L_\Omega$ as
\begin{equation}\label{defn_regional_operator}
L_\Omega u(x):=\int_{\Omega}\rho(z-x)(u(z)-u(x))dz, \quad x\in\Omega.
\end{equation}
The corresponding evolution equation is
\begin{equation}\label{eqn_non-local_heat_regional}
\left\{\begin{aligned}
    u_t(x,t)&=L_\Omega u(x,t), &x\in \Omega, t>0\\  
    u(x,0)&=u_0(x), & x\in \Omega, t=0.
    \end{aligned}
\right.
\end{equation}
In the ensuring discussion, the solution $u$ is assumed to be sufficiently regular for the evolution equation and differentiations, including under the integral, to be well-defined. First we show that the natural one-dimensional modulus of continuity conjecture on a compact interval in $\R$ is false by constructing an explicit counterexample. This counterexample is subsequently extended to higher dimensions, demonstrating the conjecture fails in general. Finally, we discuss a separate obstruction to a non-local Payne--Weinberger inequality depending only on the diameter of the domain.

\subsection{One-dimensional case}
Let $I$ denote the compact line interval $[-D/2,D/2]$. 

\begin{conj}[Non-local modulus of continuity on a compact interval]\label{thm_1d_mod_cont_reg}
Suppose that $u:I\times[0,\infty)\to\R$ satisfies the one-dimensional regional heat equation (\ref{eqn_non-local_heat_regional}) on $I$ and $\|u\|_{L^\infty(I)}<1$. Let $\varphi:[0,D/2]\times[0,\infty)\to[0,\infty)$ be a function with the following properties:
\begin{itemize}
\item[(a.)] $\varphi':=\partial\varphi(r,t)/\partial r>0$, and $\lim\limits_{r\to0^+}\varphi(r,t)=\varphi(0,t)=0$ for each $t\geq0$;
\item[(b.)] $\varphi(\cdot,0)$ is a modulus of continuity for $u(\cdot,0)$; and
\item[(c.)] $\varphi$ has the odd extension to $I$
\begin{equation}\label{2_mod_cont_odd_extn}
\tilde{\varphi}(r,t)=
\begin{cases}
    \varphi(r,t), & 0\leq r\leq D/2,\\
    -\varphi(-r,t), & -D/2\leq r<0,
\end{cases}
\end{equation}
and is supersolution of the 1d regional heat equation
\begin{equation}\label{eqn_non-localcondn_varphi}
\varphi_t(r,t)\geq L_I\tilde{\varphi}(r,t),
\qquad 0<r\leq D/2,\quad t>0.
\end{equation}
\end{itemize}
Then $\varphi(\cdot,t)$ is a modulus of continuity for $u(\cdot,t)$ on $I\times[0,\infty)$.
\end{conj}
\begin{rem} Observe that $L_I\tilde\varphi(x,t)$ is the regional operator applied to the odd
extension and is defined for every point $x\in I$.  Only the supersolution inequality (\ref{eqn_non-localcondn_varphi}) is imposed on the positive half of $I$.
\end{rem}

While this conjecture is true for symmetric pairs of points in $I$, it fails for asymmetric pairs. As we shall see, this is because the evolution of the difference between two points depends on both their distance apart and the truncation of the interaction interval by the boundary of $I$. To construct a full counterexample, we first considered the simplest possible admissible kernel, namely, the box kernel, with linear initial datum. These choices enabled us to easily compute an exact defect that could only arise due to evolution under the regional heat equation. We then mollified the box kernel to produce a full counterexample. Since there is significant technical overlap between the `proto-counterexample' using the box kernel and the full counterexample, we present only the full counterexample in what follows.\\

In order to produce a counterexample, we only need to construct a single pair of functions $(u_\eta,\varphi_\eta)$, where $u_\eta$ solves the regional heat equation for a chosen kernel $\rho_\eta$, and $\varphi_\eta$ is a comparison function satisfying properties (a.) to (c.) of the conjecture, but for which the conclusion is false.\\

\textbf{Step 1: Specify the kernel $\rho_\eta$.} We choose the specific box kernel
\begin{equation}\label{eq:regional_box_kernel}
    \rho(r):=\chi_{[-1/2,1/2]}(r)
\end{equation}
to keep calculations simple. Let $\xi$ be the non-negative, even, compactly supported smooth function used to define the standard mollifier, normalised so that $\int_{\R}\xi(r)\,dr=1$.
For $\eta>0$, set
\[
    \xi_\eta(r):=\frac1\eta\xi\left(\frac r\eta\right).
\]
Then $\xi_\eta$ is the standard mollifier; see, for example, the appendix of \cite{E}.  Define the mollification of $\rho$ by
\[
    \rho_\eta:=\rho*\xi_\eta.
\]
By the standard properties of mollification, $\rho_\eta$ is smooth, non-negative, and even. It is compactly supported in $[-1/2-\eta,1/2+\eta]$, and so trivially converges to zero at infinity. 

We verify that $\rho_\eta$ is non-increasing as a function of $|r|$ from the origin: write
\[
\begin{aligned}
    \rho_\eta(r)
    &=
    \int_{-1/2}^{1/2}\xi_\eta(r-y)\,dy=
    \int_{r-1/2}^{r+1/2}\xi_\eta(s)\,ds,
\end{aligned}
\]
making the change of variables $s=r-y$ in the second equality, and then apply the fundamental theorem of calculus to yield,
\[
    \rho_\eta'(r)
    =
    \xi_\eta\left(r+\frac12\right)
    -
    \xi_\eta\left(r-\frac12\right).
\]
For $r>0$,
\[
    \left|r+\frac12\right|
    \geq
    \left|r-\frac12\right|.
\]
Since $\xi_\eta$ is even and non-increasing as a function of $|r|$,
\[
    \xi_\eta\left(r+\frac12\right)
    \leq
    \xi_\eta\left(r-\frac12\right),
\]
and so
\[
    \rho_\eta'(r)\leq0
    \qquad\text{for }r>0,
\]
yielding the claimed result.

Furthermore, $\rho_\eta$ converges to $\rho$ in $L^1(\R)$: assume $0<\eta<1/2$.  If $|r|\leq\frac12-\eta$, then the interval $[r-1/2,r+1/2]$ contains the support of $\xi_\eta$,
and hence
\[
    \rho_\eta(r)
    =
    \int_{\R}\xi_\eta(s)\,ds
    =
    1.
\]
If instead $|r|\geq\frac12+\eta$, then the interval $[r-1/2,r+1/2]$ lies outside the interior of
$\operatorname{supp}\xi_\eta=[-\eta,\eta]$, intersecting it at most at
an endpoint where $\xi_\eta=0$. Thus,
\[
    \rho_\eta(r)=0.
\]
Moreover,
\[
    0\leq\rho_\eta(r)\leq1
    \qquad\text{for all }r\in\R.
\]
It follows that $\rho_\eta-\rho$ can be non-zero only on
\[
    \left[-\frac12-\eta,-\frac12+\eta\right]
    \cup
    \left[\frac12-\eta,\frac12+\eta\right],
\]
which has total length $4\eta$.  Since
$|\rho_\eta-\rho|\leq1$, we therefore have that
\begin{equation}
    \|\rho_\eta-\rho\|_{L^1(\R)}
    \leq
    4\eta
    \longrightarrow0
    \qquad\text{as }\eta\to0.\label{eq:regional_mollifier_L1}
\end{equation}

Thus $\rho_\eta$ is smooth, non-negative, even, non-increasing as a
function of distance, and decays to zero at infinity. It consequently belongs to the same admissible kernel class as the box kernel.\\

\textbf{Step 2: Construct a solution $u_\eta$ with suitable properties.}
Choose the linear initial datum
\[
    f(r):=\frac r2,
    \qquad r\in I,
\]
and let $L_I^\eta$ denote the regional operator with kernel $\rho_\eta$.
Let $u_\eta$ solve
\begin{equation}\label{eq:regional_mollified_evolution}
    (u_\eta)_t=L_I^\eta u_\eta,
    \qquad
    u_\eta(\cdot,0)=f.
\end{equation}
Since $\|u_\eta(\cdot,0)\|_{L^\infty(I)}=\frac12<1$, the required uniform bound holds initially.

We check that it is preserved by the evolution. As $u_\eta(\cdot,t)$ is continuous on the compact interval $I$, it
attains its spatial maximum and minimum.  If $x_t\in I$ is a point at
which the spatial maximum is attained, then
\[
    u_\eta(z,t)-u_\eta(x_t,t)\leq0
    \qquad\mbox{for every }z\in I.
\]
Since $\rho_\eta\geq0$,
\begin{align}
    (u_\eta)_t(x_t,t)
    &=
    L_I^\eta u_\eta(x_t,t)\nonumber\\
    &=
    \int_I
        \rho_\eta(z-x_t)
        \bigl(u_\eta(z,t)-u_\eta(x_t,t)\bigr)\,dz
    \leq0.
    \label{eq:regional_mollified_maximum}
\end{align}
Thus the spatial maximum cannot increase. Similarly, if $y_t\in I$ is a point at which the spatial minimum is attained, then
\[
    u_\eta(z,t)-u_\eta(y_t,t)\geq0
    \qquad\mbox{for every }z\in I,
\]
and hence, again using $\rho_\eta\geq0$,
\begin{equation}\label{eq:regional_mollified_minimum}
    (u_\eta)_t(y_t,t)
    =
    L_I^\eta u_\eta(y_t,t)
    \geq0,
\end{equation}
so the spatial minimum cannot decrease. Since
\[
    -\frac12
    \leq
    u_\eta(r,0)
    \leq
    \frac12
    \qquad(r\in I),
\]
it follows that
\begin{equation}\label{eq:regional_mollified_sup_bound}
    -\frac12
    \leq
    u_\eta(r,t)
    \leq
    \frac12,
    \qquad
    \|u_\eta(\cdot,t)\|_{L^\infty(I)}
    \leq\frac12<1
\end{equation}
for every $r\in I$ and $t\geq0$.

As $I=[-1,1]$ is symmetric, $\rho_\eta$ is even, and the initial
datum $f(r)=r/2$ is odd, oddness is preserved by the evolution.  To see this,
define
\[
    v_\eta(r,t):=-u_\eta(-r,t).
\]
A change of variables $z=-w$, together with the evenness of
$\rho_\eta$, shows that $v_\eta$ satisfies the same regional heat
equation as $u_\eta$:
\[
    (v_\eta)_t=L_I^\eta v_\eta.
\]
Moreover,
\[
    v_\eta(r,0)=-f(-r)=f(r)=u_\eta(r,0).
\]
Since $L_I^\eta$ is linear, the difference
\[
    w_\eta:=v_\eta-u_\eta
\]
satisfies
\[
    (w_\eta)_t=L_I^\eta w_\eta,
    \qquad
    w_\eta(\cdot,0)=0.
\]
Applying the maximum principle to $w_\eta$ and to $-w_\eta$ gives
$w_\eta\equiv0$.  Hence $v_\eta=u_\eta$, and therefore
\begin{equation}\label{eq:regional_mollified_oddness}
    u_\eta(-r,t)=-u_\eta(r,t),
    \qquad
    u_\eta(0,t)=0.
\end{equation}

It remains to verify that strict monotonicity is preserved at all subsequent times.  Set
\[
    g(r,t):=(u_\eta)_r(r,t).
\]
Initially,
\begin{equation}\label{eq:regional_g_initial}
    g(r,0)=\frac12>0.
\end{equation}
By the regularity assumptions on $u_\eta$, the mixed partial
derivatives commute, so differentiating
\eqref{eq:regional_mollified_evolution} with respect to $r$ gives
\[
    \partial_r (u_\eta)_t
    =
    \partial_t (u_\eta)_r
    =
    g_t
\]
on the left-hand side.  On the right-hand side, differentiation may be
passed under the integral sign as $\rho_\eta$ is smooth and,
under the same regularity assumption, the differentiated integrand is
continuous on the compact integration domain and therefore bounded by an
integrable function.  The differentiation-under-the-integral
theorem therefore applies. Thus
\begin{align}
    g_t(r,t)
    &=
    -\int_{-1}^{1}
        \rho_\eta'(z-r)
        \bigl(u_\eta(z,t)-u_\eta(r,t)\bigr)\,dz
    -q_\eta(r)g(r,t),
    \label{eq:regional_g_before_parts}
\end{align}
where
\begin{equation}\label{eq:regional_q_eta}
    q_\eta(r):=
    \int_{-1}^{1}\rho_\eta(z-r)\,dz.
\end{equation}

We now integrate the first term in
\eqref{eq:regional_g_before_parts} by parts with respect to the
integration variable $z$.  Since
\[
    \partial_z\rho_\eta(z-r)=\rho_\eta'(z-r)
    \qquad\mbox{and}\qquad
    \partial_z u_\eta(z,t)=g(z,t),
\]
we obtain
\begin{align*}
&-\int_{-1}^{1}
    \rho_\eta'(z-r)
    \bigl(u_\eta(z,t)-u_\eta(r,t)\bigr)\,dz\\
&\quad=
-\left[
    \rho_\eta(z-r)
    \bigl(u_\eta(z,t)-u_\eta(r,t)\bigr)
 \right]_{z=-1}^{z=1}
+
\int_{-1}^{1}\rho_\eta(z-r)g(z,t)\,dz.
\end{align*}
Using the evenness of $\rho_\eta$,
\[
    \rho_\eta(-1-r)=\rho_\eta(1+r),
\]
and substituting this identity into
\eqref{eq:regional_g_before_parts} yields
\begin{align}
    g_t(r,t)
    &=
    \int_{-1}^{1}\rho_\eta(z-r)g(z,t)\,dz
    -q_\eta(r)g(r,t)\nonumber\\
    &\quad
    -\rho_\eta(1-r)
       \bigl(u_\eta(1,t)-u_\eta(r,t)\bigr)\nonumber\\
    &\quad
    +\rho_\eta(1+r)
       \bigl(u_\eta(-1,t)-u_\eta(r,t)\bigr),
    \label{eq:regional_g_evolution}
\end{align}
Using
\[
    u_\eta(r,t)-u_\eta(-1,t)
    =
    \int_{-1}^{r}g(z,t)\,dz
\]
and
\[
    u_\eta(1,t)-u_\eta(r,t)
    =
    \int_{r}^{1}g(z,t)\,dz,
\]
we can rewrite \eqref{eq:regional_g_evolution} as
\begin{align}
    g_t(r,t)
    &=
    \int_{-1}^{r}
    \bigl[
        \rho_\eta(z-r)-\rho_\eta(1+r)
    \bigr]g(z,t)\,dz\nonumber\\
    &\quad+
    \int_{r}^{1}
    \bigl[
        \rho_\eta(z-r)-\rho_\eta(1-r)
    \bigr]g(z,t)\,dz
    -
    q_\eta(r)g(r,t).
    \label{eq:regional_g_rewritten}
\end{align}
For $-1\leq z\leq r$,
\[
    |z-r|=r-z\leq1+r,
\]
and for $1\geq z\geq r$,
\[
    |z-r|=z-r\leq1-r.
\]
Since $\rho_\eta$ is even and non-increasing as a function of distance,
both bracketed coefficients in
\eqref{eq:regional_g_rewritten} are non-negative.  Consequently,
whenever $g(\cdot,t)\geq0$ on $I$,
\begin{equation}\label{eq:regional_g_differential_inequality}
    g_t(r,t)\geq-q_\eta(r)g(r,t).
\end{equation}

This is sufficient to preserve \emph{strict} positivity:  suppose, to the
contrary, that $g$ does not remain strictly positive, and let $t_0>0$
be the first time at which $g$ has a zero.  Then
$g(\cdot,t)\geq0$ for $0\leq t\leq t_0$, with strict positivity for
$0\leq t<t_0$.  For each fixed
$r\in I$, \eqref{eq:regional_g_differential_inequality} gives
\[
    \frac{d}{dt}
    \left(
        e^{q_\eta(r)t}g(r,t)
    \right)
    \geq0,
    \qquad t\in[0,t_0].
\]
Integrating this differential inequality and applying \eqref{eq:regional_g_initial} leads to
\begin{equation}\label{eq:regional_g_lower_bound}
    g(r,t)
    \geq
    \frac12e^{-q_\eta(r)t}>0,\qquad t\in[0,t_0],
\end{equation}
contradicting the existence of a zero at time $t_0$.  Hence
\begin{equation}\label{eq:regional_strict_monotonicity}
    (u_\eta)_r(r,t)>0,
    \qquad r\in I,\quad t\geq0.
\end{equation}

\textbf{Step 3: Construct an appropriate comparison function $\varphi_\eta$.} Choose as the comparison function $\varphi_\eta$ the solution $u_\eta$ restricted to the positive half of $I$:
\begin{equation}\label{eq:regional_phi_eta_inside}
    \varphi_\eta(r,t):=u_\eta(r,t),
    \qquad 0\leq r\leq1.
\end{equation}
This choice is convenient because the properties already established for
$u_\eta$ directly lead to $\varphi_\eta$ satisfying properties (a.) to (c.) in the conjecture.  In particular,
\eqref{eq:regional_mollified_oddness} gives
\[
    \varphi_\eta(0,t)=u_\eta(0,t)=0,
\]
while \eqref{eq:regional_strict_monotonicity} yields
\[
    (\varphi_\eta)_r(r,t)
    =
    (u_\eta)_r(r,t)>0
    \qquad\mbox{for }0\leq r\leq1.
\]
Since $u_\eta(\cdot,t)$ is continuous on $I$, this choice also satisfies
the right-continuity requirement at the origin:
\[
    \lim_{r\to0^+}\varphi_\eta(r,t)
    =
    \lim_{r\to0^+}u_\eta(r,t)
    =
    u_\eta(0,t)
    =
    0.
\]
Thus condition~(a.) holds on the positive half of $I$.  As
$\varphi_\eta(0,t)=0$ and $\varphi_\eta$ is strictly increasing, we also
have $\varphi_\eta(r,t)>0$ for $r>0$.

Condition (b.) is also satisfied. At $t=0$,
\[
    \varphi_\eta(r,0)
    =
    u_\eta(r,0)
    =
    \frac r2.
\]
For any $x,y\in I$,
\[
\begin{aligned}
    |u_\eta(y,0)-u_\eta(x,0)|
    &=
    \frac{|y-x|}{2}=
    2\varphi_\eta\left(\frac{|y-x|}{2},0\right).
\end{aligned}
\]
It follows that $\varphi_\eta(\cdot,0)$ is a modulus of continuity for
$u_\eta(\cdot,0)$ on $I$, holding with equality for every pair.

Lastly we verify property~(c.).  On the positive half of $I$ we defined
\[
    \varphi_\eta(r,t)=u_\eta(r,t),
    \qquad 0\leq r\leq1.
\]
By preservation of oddness,
\[
    u_\eta(-r,t)=-u_\eta(r,t),
\]
so the odd extension of $\varphi_\eta$ satisfies
\[
    \tilde\varphi_\eta(r,t)=u_\eta(r,t)
    \qquad\mbox{for every }r\in I.
\]
Consequently, for $0<r\leq1$,
\begin{align}
    (\varphi_\eta)_t(r,t)=
    (u_\eta)_t(r,t)\nonumber=
    L_I^\eta u_\eta(r,t)\nonumber=
    L_I^\eta\tilde\varphi_\eta(r,t).
    \label{eq:regional_phi_eta_exact_supersolution}
\end{align}
The supersolution inequality in property~(c.) is satisfied with
equality.\\

\textbf{Step 4: Show the conjecture fails.} Select the
asymmetric pair of points
\begin{equation}\label{eq:regional_counterexample_pair}
    x=-\frac45,
    \qquad
    y=\frac35
\end{equation}
with half-separation
\begin{equation}\label{eq:regional_counterexample_half_distance}
    \frac{y-x}{2}=\frac7{10}.
\end{equation}
Define
\begin{equation}\label{eq:regional_defect_eta}
    G_\eta(t)
    :=
    u_\eta\left(\frac35,t\right)
    -
    u_\eta\left(-\frac45,t\right)
    -
    2\varphi_\eta\left(\frac7{10},t\right).
\end{equation}
Property~(b.) is attained with equality for every pair at $t=0$, so
\begin{equation}\label{eq:regional_defect_eta_initial}
    G_\eta(0)=0.
\end{equation}
Since property~(c.) is also satisfied with equality,
\[
    G_\eta'(0)
    =
    L_I^{\rho_\eta}f\left(\frac35\right)
    -
    L_I^{\rho_\eta}f\left(-\frac45\right)
    -
    2L_I^{\rho_\eta}f\left(\frac7{10}\right).
\]
For a kernel $\sigma$, set
\begin{equation}\label{eq:regional_defect_functional}
    \mathcal D(\sigma)
    :=
    L_I^\sigma f\left(\frac35\right)
    -
    L_I^\sigma f\left(-\frac45\right)
    -
    2L_I^\sigma f\left(\frac7{10}\right).
\end{equation}
Thus
\[
    G_\eta'(0)=\mathcal D(\rho_\eta).
\]

We now compute the reference value $\mathcal D(\rho)$ for the box kernel
\eqref{eq:regional_box_kernel}.  Since
\[
    L_I^\rho f(p)
    =
    \int_{I\cap[p-1/2,p+1/2]}
        \left(\frac z2-\frac p2\right)\,dz,
\]
we obtain, at $p=3/5$,
\begin{equation}
    L_I^\rho f\left(\frac35\right)=
    \int_{1/10}^{1}
        \left(\frac z2-\frac3{10}\right)\,dz\nonumber=
    \left[
        \frac{z^2}{4}-\frac{3z}{10}
    \right]_{1/10}^{1}
    =
    -\frac9{400}.
    \label{eq:regional_L_three_fifths}
\end{equation}
At $p=-4/5$,
\begin{equation}
    L_I^\rho f\left(-\frac45\right)=
    \int_{-1}^{-3/10}
        \left(\frac z2+\frac25\right)\,dz\nonumber=
    \left[
        \frac{z^2}{4}+\frac{2z}{5}
    \right]_{-1}^{-3/10}
    =
    \frac{21}{400}.
    \label{eq:regional_L_minus_four_fifths}
\end{equation}
Finally, at $p=7/10$,
\begin{equation}
    L_I^\rho f\left(\frac7{10}\right)=
    \int_{1/5}^{1}
        \left(\frac z2-\frac7{20}\right)\,dz\nonumber=
    \left[
        \frac{z^2}{4}-\frac{7z}{20}
    \right]_{1/5}^{1}
    =
    -\frac1{25}
    =
    -\frac{16}{400}.
    \label{eq:regional_L_seven_tenths}
\end{equation}
Consequently,
\begin{align}
    \mathcal D(\rho)
    &=
    -\frac9{400}
    -
    \frac{21}{400}
    +
    \frac{32}{400}\nonumber\\
    &=
    \frac1{200}>0.
    \label{eq:regional_positive_defect}
\end{align}

It remains only to transfer this strict positivity from the box kernel
to its smooth mollification.  For any $p\in I$,
\begin{align}
    \left|
        L_I^{\rho_\eta}f(p)-L_I^\rho f(p)
    \right|
    &=
    \left|
        \int_I
        \bigl(\rho_\eta(z-p)-\rho(z-p)\bigr)
        \bigl(f(z)-f(p)\bigr)\,dz
    \right|\nonumber\\
    &\leq
    \int_I
        |\rho_\eta(z-p)-\rho(z-p)|
        |f(z)-f(p)|\,dz\nonumber\\
    &\leq
    \int_I
        |\rho_\eta(z-p)-\rho(z-p)|\,dz\nonumber\\
    &\leq
    \|\rho_\eta-\rho\|_{L^1(\R)},
    \label{eq:regional_operator_kernel_continuity}
\end{align}
because $|f(z)-f(p)|\leq1$ on $I$.  Applying this estimate to the three
terms in \eqref{eq:regional_defect_functional} and using the triangle
inequality gives
\begin{align}
    \left|
        \mathcal D(\rho_\eta)-\mathcal D(\rho)
    \right|
    &\leq
    4\|\rho_\eta-\rho\|_{L^1(\R)}\nonumber\leq
    16\eta,
    \label{eq:regional_defect_mollification}
\end{align}
where the final inequality follows from
\eqref{eq:regional_mollifier_L1}.  Hence,
\[
    \mathcal D(\rho_\eta)
    \geq
    \frac1{200}-16\eta,
\]
and therefore, for example,
\begin{equation}\label{eq:regional_explicit_eta_choice}
    0<\eta<\frac1{3200}
    \quad\Longrightarrow\quad
    \mathcal D(\rho_\eta)>0.
\end{equation}

Fix such an $\eta$.  Then
\[
    G_\eta'(0)=\mathcal D(\rho_\eta)>0.
\]
Together with \eqref{eq:regional_defect_eta_initial}, differentiability
at $t=0$ gives
\[
    G_\eta(t)
    =
    \mathcal D(\rho_\eta)t+o(t)>0
\]
for all sufficiently small $t>0$.  Since
$u_\eta(\cdot,t)$ is strictly increasing and $3/5>-4/5$,
\[
    u_\eta\left(\frac35,t\right)
    -
    u_\eta\left(-\frac45,t\right)>0.
\]
Thus, for all sufficiently small positive $t$,
\[
    \left|
        u_\eta\left(\frac35,t\right)
        -
        u_\eta\left(-\frac45,t\right)
    \right|
    =
    u_\eta\left(\frac35,t\right)
    -
    u_\eta\left(-\frac45,t\right)
    >
    2\varphi_\eta\left(\frac7{10},t\right).
\]
We see that the smooth mollified construction therefore satisfies all the
hypotheses of Conjecture~\ref{thm_1d_mod_cont_reg} while violating its
conclusion.\\

The fact that $\mathcal D(\rho)>0$ comes from the boundary truncation of the
interaction neighbourhoods.  At the three points used above, the
effective regions of integration for the box kernel are respectively
\[
    \left[\frac1{10},1\right],
    \qquad
    \left[-1,-\frac3{10}\right],
    \qquad
    \left[\frac15,1\right].
\]
Thus the evolution of a difference between two points is not determined
by their separation alone: their positions relative to the boundary
also enter through the truncation of their interaction neighbourhoods.

\subsection{Higher-dimensional case}\label{sec_regional_gen}

We next formulate the corresponding modulus-of-continuity conjecture
on bounded convex subsets of $\R^n$.

\begin{conj}[Non-local modulus of continuity on a bounded convex set]
\label{thm_nd_mod_cont_reg}
Let $\Omega$ be a bounded convex set in $\R^n$ with diameter $D$, and
set $I=[-D/2,D/2]$.  Choose rectangular coordinates
$(q,p)\in\R\times\R^{n-1}$ and, if necessary, translate the origin
along the $q$-axis so that the orthogonal projection of $\Omega$ onto
the $q$-axis is contained in $I$.  For each $q\in\R$, define the
transverse slice
\begin{equation}\label{eqn_w_slice_Omega}
    \Omega_q
    :=
    \{p\in\R^{n-1}:(q,p)\in\Omega\},
\end{equation}
with $\Omega_q=\varnothing$ when the slice is empty, and define the
induced one-dimensional kernel
\begin{equation}\label{eqn_induced_kernel_regional}
    \tilde\rho(q)
    :=
    \int_{\Omega_q}\rho(q,p)\,dp,
    \qquad q\in\R.
\end{equation}

Suppose that $u:\Omega\times[0,\infty)\to\R$ satisfies the regional
heat equation~\eqref{eqn_non-local_heat_regional} and
$\|u\|_{L^\infty(\Omega)}<1$.  Let
$\varphi:[0,D/2]\times[0,\infty)\to[0,\infty)$ be a function with the
following properties:
\begin{itemize}
\item[(a.)] $\varphi_r(r,t)>0$, and
\[
    \lim_{r\to0^+}\varphi(r,t)
    =
    \varphi(0,t)
    =
    0
\]
for each $t\geq0$;

\item[(b.)] $\varphi(\cdot,0)$ is a modulus of continuity for
$u(\cdot,0)$ on $\Omega$; and

\item[(c.)] $\varphi$ has the odd extension to $I$
\begin{equation}\label{eq:nd_odd_extension}
\tilde\varphi(r,t)
=
\begin{cases}
    \varphi(r,t), & 0\leq r\leq D/2,\\
    -\varphi(-r,t), & -D/2\leq r<0,
\end{cases}
\end{equation}
satisfying
\begin{equation}\label{eqn_non-localcondn_varphi_general2}
    \varphi_t(r,t)
    \geq
    \int_I
        \tilde\rho(w-r)
        \bigl(
            \tilde\varphi(w,t)-\tilde\varphi(r,t)
        \bigr)\,dw,
    \qquad
    0<r\leq D/2,\quad t>0.
\end{equation}
\end{itemize}

Then, for every $t\geq0$, $\varphi(\cdot,t)$ is a modulus of
continuity for $u(\cdot,t)$ on $\Omega$.
\end{conj}

\begin{prop}\label{prop:nd_counterexample}
Conjecture~\ref{thm_nd_mod_cont_reg} is false for every $n\geq2$.
\end{prop}

The same obstruction found in the one-dimensional case persists in
every dimension $n\geq2$.  For $0<\epsilon\ll1$, let
$\Omega_\epsilon$ be a thin convex cylinder of diameter $2$, and, as
before, choose linear initial data $F(q,p)=\frac{q}{2}$, depending only
on the longitudinal variable $q$.  Integrating the higher-dimensional analogue of the box kernel over each transverse slice of $\Omega_\epsilon$
produces an induced one-dimensional kernel $\tilde\rho_\epsilon$.  If
$c_\epsilon$ denotes the $(n-1)$-dimensional volume of a transverse
slice, then the normalised kernel
\[
    k_\epsilon:=\frac{\tilde\rho_\epsilon}{c_\epsilon}
\]
converges in $L^1(\R)$, as $\epsilon\to0$, to the one-dimensional box
kernel
\[
    k_0=\chi_{[-1/2,1/2]}.
\]
After division by $c_\epsilon$, the derivative at $t=0$ of the
corresponding modulus defect consequently converges to the
one-dimensional value
\[
    \frac1{200}>0.
\]
Its initial derivative is therefore positive for all sufficiently thin
cylinders, which implies the defect itself is positive for sufficiently
small $t>0$.  This proves Proposition~\ref{prop:nd_counterexample}.
Thus the same boundary-truncation mechanism occurs on $n$-dimensional
convex sets. The full calculation is
given in Appendix~\ref{app:higher_dimensional_defect}.


\subsection{Counterexample to a classical nonlocal Payne-Weinberger inequality}
It is well known that the regional fractional Laplacian has a discrete increasing spectrum of non-negative eigenvalues with corresponding eigenfunctions that form an orthonormal basis for $L^2(\Omega)$.
Since the first eigenvalue is simple, taking value zero with an eigenfunction that can be chosen to be $1$, the first non-trivial eigenvalue $\lambda_2$ can be defined using a variational formulation by
\begin{align}\label{lambda2_variational}
    \lambda_2&=\min_{\substack{u\in H^s(\Omega)\setminus\{0\}\\ \langle u, 1\rangle_{L^2(\Omega)}=\int_\Omega u=0}}\bigg\{\frac{\mathcal{E}_{reg}[u,u]}{\|u\|_{L^2(\Omega)}^2}\bigg\},
\end{align}
where $u$ is an admissible function, $\mathcal{E}_{reg}[\cdot,\cdot]$ is defined as
\begin{align*}
    \mathcal{E}_{reg}[u,u]=\frac{c_n}{2}\iint\limits_{\Omega\times\Omega}\frac{(u(x)-u(y))^2}{|x-y|^{n+2s}}dxdy,
\end{align*} 
and
\begin{align*}
    H^s(\Omega)=\left\{u\in L^2(\Omega): \iint\limits_{\Omega\times\Omega}\frac{(u(x)-u(y))^2}{|x-y|^{n+2s}}dxdy<\infty\right\}
\end{align*}
We show that the classical Payne-Weinberger inequality cannot be extended to the regional fractional Laplace operator through the following counterexample.

Fix $D>0$. For $0<L<D$, define
\[
\epsilon_L:=\frac{1}{2}\sqrt{D^2-L^2}
\qquad\text{and}\qquad
\Omega_L:=[0,L]\times[-\epsilon_L,\epsilon_L].
\]
Observe that $\operatorname{diam}(\Omega_L)=D$ for every $L$, while $\epsilon_L\to D/2$ as $L\to0^+$. Let
\[
u(x,y):=y.
\]
Then
\begin{align}
\int_{\Omega_L}u(x,y)\,dx\,dy
=\int_0^L\int_{-\epsilon_L}^{\epsilon_L}y\,dy\,dx
=0.
\end{align}
Thus $u$ has mean zero. Since $u$ is Lipschitz on the compact set, it also belongs to $H^s(\Omega_L)$. It follows that $u$ is an admissible test function in equation (\ref{lambda2_variational}).

Set $p=(x_1,y_1)$ and $q=(x_2,y_2)$.  Then
\begin{align}
\mathcal E_{\mathrm{reg}}[u,u]
&=
\frac{c_{2,s}}{2}
\int_0^L\int_0^L
\int_{-\epsilon_L}^{\epsilon_L}
\int_{-\epsilon_L}^{\epsilon_L}
\frac{(y_1-y_2)^2}
{\bigl((x_1-x_2)^2+(y_1-y_2)^2\bigr)^{1+s}}
\,dy_1\,dy_2\,dx_1\,dx_2.
\end{align}

Fix $y_1,y_2$ in $[0,L]$. Make the
change of variables
\[
    h:=x_1-x_2,
    \qquad
    r:=x_2,
\]
so that $x_1=h+r$, $x_2=r$ with unit Jacobian.  For each
$h\in[-L,L]$, the conditions $0\leq x_1,x_2\leq L$ restrict $r$ to
an interval of length $L-|h|$.  Hence,
\[
    \int_0^L\int_0^L
        G(x_1-x_2)\,dx_1\,dx_2
    =
    \int_{-L}^L
        (L-|h|)G(h)\,dh.
\]
Likewise, setting
\[
    k:=y_1-y_2,
    \qquad
    v:=y_2,
\]
the Jacobian is again one, and for fixed
$k\in[-2\epsilon_L,2\epsilon_L]$ the integration interval in $v$ has
length $2\epsilon_L-|k|$.  Therefore,
\[
    \int_{-\epsilon_L}^{\epsilon_L}
    \int_{-\epsilon_L}^{\epsilon_L}
        H(y_1-y_2)\,dy_1\,dy_2
    =
    \int_{-2\epsilon_L}^{2\epsilon_L}
        (2\epsilon_L-|k|)H(k)\,dk.
\]

Applying these changes of variables successively gives
\begin{align}
\mathcal E_{\mathrm{reg}}[u,u]
&=
\frac{c_{2,s}}{2}
\int_{-L}^{L}
\int_{-2\epsilon_L}^{2\epsilon_L}
(L-|h|)(2\epsilon_L-|k|)
\frac{k^2}{(h^2+k^2)^{1+s}}
\,dk\,dh.
\end{align}
Our specific integrand is even separately in $h$ and $k$, and $L-|h|\leq L$ and $2\epsilon_L-|k|\leq 2\epsilon_L$. Thus
\begin{align}
\mathcal E_{\mathrm{reg}}[u,u]
&=
2c_{2,s}
\int_0^L
\int_0^{2\epsilon_L}
(L-h)(2\epsilon_L-k)
\frac{k^2}{(h^2+k^2)^{1+s}}
\,dk\,dh \nonumber\\
&\leq
4c_{2,s}L\epsilon_L
\int_0^L
\int_0^{2\epsilon_L}
\frac{k^2}{(h^2+k^2)^{1+s}}
\,dk\,dh.
\label{reg_counterex1}
\end{align}

For $L$ sufficiently small, $L<2\epsilon_L$.  Split the integration region of the last
integral along the line $k=h$ into the regions $0<k<h$ and
$h<k<2\epsilon_L$.  On the first region,
\[
    \frac{k^2}{(h^2+k^2)^{1+s}}
    \leq
    k^2h^{-2-2s},
\]
and hence
\begin{align}
\int_0^L\int_0^h
\frac{k^2}{(h^2+k^2)^{1+s}}\,dk\,dh
&\leq
\frac13\int_0^L h^{1-2s}\,dh=
\frac{L^{2-2s}}{6(1-s)}.
\label{eq:counterexample_region_one}
\end{align}

Over the second region, we have
\[
    \frac{k^2}{(h^2+k^2)^{1+s}}
    \leq
    k^{-2s}.
\]
Therefore
\[
    \int_0^L\int_h^{2\epsilon_L}
    \frac{k^2}{(h^2+k^2)^{1+s}}\,dk\,dh
    \leq
    \int_0^L\int_h^{2\epsilon_L} k^{-2s}\,dk\,dh.
\]
The last integral depends on whether $1-2s$ is positive, zero, or
negative.  If $0<s<1/2$, then
\[
\begin{aligned}
\int_0^L\int_h^{2\epsilon_L} k^{-2s}\,dk\,dh
&\leq
\frac{L(2\epsilon_L)^{1-2s}}{1-2s}.
\end{aligned}
\]
If $s=1/2$, then
\[
\begin{aligned}
\int_0^L\int_h^{2\epsilon_L} \frac{1}{k}\,dk\,dh
&=
\int_0^L
\log\left(\frac{2\epsilon_L}{h}\right)\,dh=
L\left(1+\log\frac{2\epsilon_L}{L}\right).
\end{aligned}
\]
Finally, if $1/2<s<1$, then
\[
\begin{aligned}
\int_0^L\int_h^{2\epsilon_L} k^{-2s}\,dk\,dh
&\leq
\frac{1}{2s-1}\int_0^L h^{1-2s}\,dh=
\frac{L^{2-2s}}
{2(2s-1)(1-s)}.
\end{aligned}
\]

Since $L<2\epsilon_L$, for $0<s<1/2$ we also have
\[
    L^{2-2s}
    \leq
    L(2\epsilon_L)^{1-2s}.
\]
Summing over the two integration regions then gives
\begin{equation}\label{eq:counterexample_integral_bound}
\int_0^L\int_0^{2\epsilon_L}
\frac{k^2}{(h^2+k^2)^{1+s}}\,dk\,dh
\leq
\begin{cases}
\displaystyle
\left(
\frac{1}{6(1-s)}+\frac{1}{1-2s}
\right)
L(2\epsilon_L)^{1-2s},
& 0<s<\frac12,\\[4mm]
\displaystyle
L\left(
\frac43+\log\frac{2\epsilon_L}{L}
\right),
& s=\frac12,\\[4mm]
\displaystyle
\left(
\frac{1}{6(1-s)}
+
\frac{1}{2(2s-1)(1-s)}
\right)
L^{2-2s},
& \frac12<s<1.
\end{cases}
\end{equation}
Moreover,
\begin{align}
\|u\|_{L^2(\Omega_L)}^2
&=
\int_0^L\int_{-\epsilon_L}^{\epsilon_L}y^2\,dy\,dx
=
\frac{2L\epsilon_L^3}{3}.
\label{eq:counterexample_L2_norm}
\end{align}
Combining the variational characterisation,
\eqref{reg_counterex1}, and
\eqref{eq:counterexample_integral_bound}, we obtain
\begin{equation}\label{eq:counterexample_rayleigh_bound}
0<\lambda_2
\leq
6c_{2,s}\epsilon_L^{-2}
\begin{cases}
\displaystyle
\left(
\frac{1}{6(1-s)}+\frac{1}{1-2s}
\right)
L(2\epsilon_L)^{1-2s},
& 0<s<\frac12,\\[4mm]
\displaystyle
L\left(
\frac43+\log\frac{2\epsilon_L}{L}
\right),
& s=\frac12,\\[4mm]
\displaystyle
\left(
\frac{1}{6(1-s)}
+
\frac{1}{2(2s-1)(1-s)}
\right)
L^{2-2s},
& \frac12<s<1.
\end{cases}
\end{equation}

Since
\[
    \epsilon_L\longrightarrow\frac D2>0
    \qquad\text{as }L\to0^+,
\]
every quantity on the right-hand side tends to zero.  Indeed,
$L\log(1/L)\to0$ in the borderline case $s=1/2$, while
$2-2s>0$ for $s<1$.  Hence,
\[
    \lambda_2(\Omega_L)\longrightarrow0.
\]
On the other hand,
\[
    \operatorname{diam}(\Omega_L)=D\quad\text{for every }L.
\]  
If there were a positive lower bound
$B(D)>0$ depending only on the diameter, then
\[
    B(D)\leq\lambda_2(\Omega_L)
\]
for every $L$.  Passing to the limit $L\to0^+$ would give
\[
    B(D)\leq0,
\]
contradicting $B(D)>0$. It follows that there is no positive lower bound for
$\lambda_2$ depending only on the diameter of the domain.

\section{Non-linear non-local heat equation}\label{sec_non_linear_non_local_heat_eqn}
The methods of the proof of Theorem \ref{thm_modulus_Rn} can also be adapted to prove a modulus of continuity theorem for solutions $u:\R^n\times [0,\infty)\to\R$ of the non-linear non-local heat equation
\begin{equation}\label{eqn_nonlin_fractionalheat_Rn}
 \left\{\begin{aligned}
        u_t(x,t)&=\Lop u(x,t)+q(|Du(x,t)|), & x\in\R^n, t\in(0,\infty)\\
        u(x,0)&=u_0(x), & x\in\R^n,
       \end{aligned}
 \right.
\end{equation}
given suitable constraints on the nonlinear term $q$. 

\begin{thm}[Non-linear non-local modulus of continuity on $\R^n$]\label{thm_mod_cont_nonlin}
Suppose $u:\R^n\times[0,\infty)\to\R$ is a smooth solution of equation (\ref{eqn_nonlin_fractionalheat_Rn}) and $\|u\|_{L^\infty(\R^n)}\leq 1$. The non-linear term $q(p)$ (for $p\ge 0$) satisfies the following conditions:
\begin{itemize}
    \item[(i.)] $q(p)$ is non-decreasing and $q(0)=0$; and
    \item[(ii.)] $q(p)$ is at least $C^1$ and satisfies the Lipschitz condition $\left|\frac{d q(p)}{d p}\right| \le c$ for some constant $c>0$.
\end{itemize}
Let $\varphi:[0,\infty)\times[0,\infty)\to[0,\infty)$ be a function with the following properties:
\begin{enumerate}
\item[(a.)] $\varphi':=\partial_r\varphi(r,t)>0$ for $r>0$ and $\lim\limits_{r\to 0^+}\varphi(r,t)=\varphi(0,t)=0$ for each $t\geq 0$.
\item[(b.)] $\varphi(\cdot,0)$ is a modulus of continuity for $u(\cdot,0)$.
\item[(c.)] $\varphi$ has an odd extension $\tilde{\varphi}$ satisfying the one-dimensional supersolution inequality
\begin{equation}\label{eqn_non-localcondn_varphi_2}
\tilde\varphi_t(r,t)\geq \LopOneD\tilde{\varphi}(r,t),
\end{equation}
where $\LopOneD$ is the one-dimensional reduction of $\Lop$.
\end{enumerate}
Furthermore, assume we can define a suitable regularisation function:
\begin{enumerate}
\item[(d.)] There exists a smooth, radially symmetric function $\Psif:\R^n \to [1,\infty)$ such that $\Psif(x) \to \infty$ as $|x|\to\infty$, satisfying:
    \begin{equation}\label{eq:K_bounds_nonlinear}
        0\le K := \sup_{x\in\R^n} \Lop\Psif(x) < \infty \quad \text{and} \quad K_\nabla := \sup_{x\in\R^n} \frac{|D\Psif(x)|}{\Psif(x)} < \infty.
    \end{equation}
\end{enumerate}
Then $\varphi(\cdot,t)$ is a modulus of continuity for $u(\cdot,t)$ on $\R^n\times[0,\infty)$.
\end{thm}

\begin{proof} 
We want to show that $|u(y,t)-u(x,t)|\leq 2\varphi\left(\frac{|x-y|}{2},t\right)$ for all $x,y\in\R^n$ and $t\geq 0$.\\

\textbf{Step 1: Define the auxiliary function.} 
Let $K$ and $K_\nabla$ be the constants from condition (d.). Choose the time constant $C > K + c K_\nabla$. We define the auxiliary function $Z_{\epsilon}:\R^n\times\R^n\times[0,\infty)\to\R$:
\begin{align*}
    Z_\epsilon(x,y,t)=u(y,t)-u(x,t)-2\varphi\left(\frac{|x-y|}{2},t\right)-\epsilon e^{Ct}(\Psif(x)+\Psif(y)).
\end{align*}
It suffices to show $Z_\epsilon<0$. 

The cases $x=y$ and $t=0$ follow from conditions (a.) and (b.), respectively. Since $\|u\|_\infty \le 1$ and $\Psi(x) \to \infty$ as $|x|\to\infty$, $Z_\epsilon \to -\infty$ as $|x|$ or $|y| \to \infty$. Therefore, $Z_\epsilon$ must attain a global maximum at some finite point.\\

\textbf{Step 2: Maximum principle and time derivative.}
To produce a contradiction, assume that for some $\epsilon_0>0$ there exists a first time $t_0>0$ and distinct points $x_0, y_0$ such that $Z_{\epsilon_0}(x_0,y_0,t_0)=0$ (a global maximum), and $Z_{\epsilon_0} \le 0$ globally for $t\le t_0$.

Simplify notation: Let $\epsilon=\epsilon_0$, $x=x_0, y=y_0, t=t_0$, $s=|x-y|/2$, and $E = \epsilon e^{Ct}$. Also, suppress time dependence in $u$ and $\varphi$.

At the maximum point, the time derivative is non-negative. Substituting the equations for $u_t$ and $\varphi_t$ gives
\begin{align*}
    0\leq \frac{\partial Z_{\epsilon}}{\partial t} =& (\Lop u(y)+q(|Du(y)|)) - (\Lop u(x)+q(|Du(x)|)) \\
    &- 2\varphi_t(s) - CE(\Psif(x)+\Psif(y)).
\end{align*}
Rearranging yields
\begin{equation}\label{eqn_Zt_nonlinear_main}
    0\leq (\Lop u(y) - \Lop u(x)) + T_q - 2\varphi_t(s) - C E(\Psif(x)+\Psif(y)).
\end{equation}
Here $T_q = q(|Du(y)|) - q(|Du(x)|)$.\\

\textbf{Step 3: Coupling-by-reflection and maximum principle bounds.} 
We employ the coupling-by-reflection technique and the maximum principle as used in the proof of Theorem (\ref{thm_modulus_Rn}) to find an upper bound for the linear spatial terms. The details are exactly the same (Steps 3 and 4), so we assume this is done, and arrive at the following inequality:
\begin{align}\label{nonlinear_final_inequality}
    0&\leq (\Lop u(y) - \Lop u(x)- 2\varphi_t) + T_q  - C E(\Psif(x)+\Psif(y))\nonumber\\
    &\leq (2\tilde\varphi_t(s)+2EK-2\varphi_t(s))+T_q-CE(\Psi(x)+\Psi(y))\nonumber\\
    &=2EK+T_q-CE(\Psi(x)+\Psi(y))\nonumber\\
    &=2EK+T_q-CP(x,y),
\end{align}
where $P(x,y) = E(\Psif(x)+\Psif(y))$ denotes the penalty term.\\

\textbf{Step 4: Analyse the nonlinear term $T_q$.} Let $e_1=\frac{y-x}{|y-x|}$. By the Lipschitz condition (ii.):
$$ T_q = q(|Du(y)|) - q(|Du(x)|) \le c ||Du(y)| - |Du(x)||. $$
We calculate the spatial derivative condition $\nabla Z_\epsilon = 0$ at the maximum point:
\begin{align*}
Du(x) &= \varphi' e_1 - E D\Psif(x), \quad Du(y) = \varphi' e_1 + E D\Psif(y).
\end{align*}
Using the above spatial derivatives conditions, the triangle inequality, and the gradient growth bound from (d.), $|D\Psif(x)| \le K_\nabla \Psif(x)$, we derive the following upper bound:
\begin{align*}
||Du(y)| - |Du(x)|| &\leq |Du(y)-Du(x)|\\
&=E|D\Psi(y)+D\Psi(x)|\\
&\le E \left(|D\Psif(y)| + |D\Psif(x)|\right) \\
&\le K_\nabla E (\Psif(y) + \Psif(x))\\
&=K_\nabla P(x,y).
\end{align*}
Thus, $T_{q} \le c K_\nabla P(x,y)$.\\

\textbf{Step 5: The contradiction.}
Substitute the $T_q$ upper bound from Step 4 into the final inequality \eqref{nonlinear_final_inequality} in Step 3:
\begin{align*}
    0&\leq 2EK+T_q-CP(x,y)\\
    &\le 2EK + c K_\nabla P(x,y) - C P(x,y).
\end{align*}
Since $\Psif \ge 1$, $ P(x,y)\ge 2E$. We assume that $K\ge 0$, so $2EK \le KP(x,y)$. This gives
\begin{align*}
    0 &\le K P(x,y) + c K_\nabla P(x,y) - C P(x,y)= P(x,y) (K + c K_\nabla - C).
\end{align*}
As $P(x,y)>0$ and we chose $C > K + c K_\nabla$, the term in parentheses is strictly negative. This yields the contradiction.

Hence, $Z_{\epsilon}<0$ for each $\epsilon>0$. Letting $\epsilon \to 0$ completes the proof.

\end{proof}

\begin{rem}
The same argument extends to principal-value kernels satisfying the Lévy condition, including the fractional Laplacian. More precisely, suppose that $\Lop$ is defined in the principal-value sense by
\[
\Lop u(x)=\PV\int_{\R^n}K(r)\bigl(u(x+r)-u(x)\bigr)\,dr,
\]
where $K:\R^n\setminus\{0\}\to[0,\infty)$ is even, radial, non-increasing in $|r|$, and satisfies
\[
\int_{\R^n}(1\wedge |r|^2)\,K(r)\,dr<\infty.
\]
Assume also that the odd extension $\tilde\varphi$ is a supersolution of the corresponding one-dimensional principal-value equation for the reduced kernel $\tilde K$, and that the regularisation function $\Psif$ satisfies the same bounds as in \eqref{eq:K_bounds_nonlinear}, with $\Lop\Psif$ interpreted in the principal-value sense.

Then the proof above goes through with the same modifications as in the proof of Theorem~\ref{thm:pv_modulus_continuity}: one replaces the linear terms $\Lop u(y)-\Lop u(x)$ by truncated operators, applies the truncated reflection identity, and lets the truncation parameter tend to zero. The only new contribution is the same small-ball remainder term, which vanishes in the limit exactly as in Theorem~\ref{thm:pv_modulus_continuity}. The nonlinear estimate for
\[
T_q=q(|Du(y)|)-q(|Du(x)|)
\]
is unchanged, since it depends only on the first-order conditions at the maximum point and not on the absolute convergence of the nonlocal operator.

In particular, the theorem extends to the case
\[
\Lop=-(-\Delta)^s,\qquad s\in(0,1),
\]
for which
\[
K(r)=C_{n,s}|r|^{-n-2s}
\qquad\mbox{and}\qquad
\tilde K(w)=C_{1,s}|w|^{-1-2s},
\]
provided the corresponding principal-value version of condition \textup{(c.)} holds.
\end{rem}

\appendix
\section{Application of Theorem~\ref{thm_modulus_Rn} to the fractional Laplacian}
\label{app:barrier}

While the main result of this paper (Theorem \ref{thm_modulus_Rn}) is established for a general class of non-local operators, our primary motivation is the fractional Laplacian. In this appendix, we verify that the standard fractional Laplacian admits a barrier function satisfying the requisite global boundedness condition (d.).\\

The specific regularisation function of interest is 
\begin{equation}\label{fn_reg_frac_lap}
    \Psi(x) = (1+|x|^2)^{\alpha/2}, \qquad 0<\alpha<2s,\qquad s\in(0,1).
\end{equation}
This function is smooth on $\R^n$
for all values of $\alpha$ and $s$. It is also radial, belongs to $C^{1,1}_{\mathrm{loc}}(\R^n)$, satisfies $\Psi\geq 1$, and satisfies $\Psi(x)\to\infty$ as $|x|\to\infty$. Applying the operator $\Lop$ to $\Psi(x)$ gives
\begin{equation}\label{frac_lap_psi}
    \Lop\Psi(x)=-(-\Delta)^s \Psi(x)= C_{n,s} \, \PV\int_{\mathbb{R}^n} \frac{\Psi(y) - \Psi(x)}{|y-x|^{n+2s}} \, dy.
\end{equation}
To facilitate the calculations to follow, we will rewrite $L_{\R^n}\Psi(x)$ in a symmetric form. 
 
First, for $\varepsilon>0$, define the truncated operator
\begin{equation}\label{frac_lap_trunc}
    \Lop^{\varepsilon}\Psi(x):=
    C_{n,s}\int_{|y-x|>\varepsilon}\frac{\Psi(y)-\Psi(x)}{|y-x|^{n+2s}}\,dy.
\end{equation}
Then
\[
    \Lop\Psi(x)=\lim_{\varepsilon\downarrow 0}\Lop^{\varepsilon}\Psi(x)
\]
whenever the principal value exists. Let $y=x+r$. Then
\begin{equation}\label{frac_lap_trunc_r}
    \Lop^{\varepsilon}\Psi(x)=
    C_{n,s}\int_{|r|>\varepsilon}\frac{\Psi(x+r)-\Psi(x)}{|r|^{n+2s}}\,dr.
\end{equation}
Now consider the change of variables $\rho=-r$, which yields
\begin{equation}\label{frac_lap_trunc_minus_r}
    \Lop^{\varepsilon}\Psi(x)=
    C_{n,s}\int_{|\rho|>\varepsilon}\frac{\Psi(x-\rho)-\Psi(x)}{|\rho|^{n+2s}}\,d\rho.
\end{equation}
Using the norm property $|-\rho|=|\rho|$ and renaming the dummy variable $\rho$ back to $r$, we have
\[
    \Lop^{\varepsilon}\Psi(x)=
    C_{n,s}\int_{|r|>\varepsilon}\frac{\Psi(x-r)-\Psi(x)}{|r|^{n+2s}}\,dr.
\]
We now express $\Lop^{\varepsilon}\Psi(x)$ as the average of these two equivalent forms:
\begin{align*}
    \Lop^{\varepsilon}\Psi(x)
    &=\frac{C_{n,s}}{2}\int_{|r|>\varepsilon}
    \frac{[\Psi(x+r)-\Psi(x)]+[\Psi(x-r)-\Psi(x)]}{|r|^{n+2s}}\,dr\\
    &=\frac{C_{n,s}}{2}\int_{|r|>\varepsilon}
    \frac{\Psi(x+r)+\Psi(x-r)-2\Psi(x)}{|r|^{n+2s}}\,dr.
\end{align*}

Thus, after passing to the limit $\varepsilon\downarrow 0$,
\begin{equation}\label{symm_form}
    \Lop\Psi(x)=\frac{C_{n,s}}{2} \, \PV \int_{\mathbb{R}^n} \frac{\Psi(x+r) + \Psi(x-r)-2\Psi(x)}{|r|^{n+2s}} \, dr.
\end{equation}

Since $\Psi$ is smooth, we may Taylor expand $\Psi(x\pm r)$ about $x$ up to second order, and show that the integral in (\ref{symm_form}) converges absolutely: let $\delta(x,r)$ $=\Psi(x+r) + \Psi(x-r)-2\Psi(x)$. Near the singularity at the origin, $\delta(x,r)=O(|r|^2)$. The integrand is consequently bounded by $O(|r|^{2-2s-n})$ near the origin, which is Lebesgue-integrable provided that $s<1$. Away from the origin, $\delta(x,r)=O(|r|^\alpha)$. The integrand is thus bounded by $O(|r|^{\alpha-2s-n})$, and it is Lebesgue-integrable for $\alpha < 2s$. As we require $0<\alpha<2s$ and $0<s<1$, the integral in equation (\ref{symm_form}) converges absolutely. This allows us to drop the principal value, giving the symmetric integral form
\begin{equation}
 \Lop\Psi(x)=\frac{C_{n,s}}{2} \int_{\mathbb{R}^n} \frac{\Psi(x+r) + \Psi(x-r)-2\Psi(x)}{|r|^{n+2s}} \, dr,
\end{equation}
which we use throughout the remainder of this appendix.\\

We now show that the function $F(x) = \Lop{\Psi}(x)$ is globally bounded on $\R^n$.

\begin{prop}\label{prop_frac_lap_reg_fn} If the parameter $\alpha$ satisfies $0 < \alpha < 2s$, then the function $F(x) = \Lop{\Psi}(x)$ is globally bounded on $\R^n$, and the global bound is non-negative.
\end{prop}
\begin{proof} The strategy is to prove that $F(x)$ is continuous on $\R^n$ and vanishes at infinity. A standard theorem in analysis states that a continuous function on $\R^n$ that vanishes at infinity is globally bounded.\\

\textbf{Step 1: Continuity.} We employ the sequential characterisation of continuity and the Dominated Convergence Theorem (DCT). Let $\{x_k\}_{k=1}^\infty$ be a sequence in $\mathbb{R}^n$ converging to an arbitrary point $x_0 \in \mathbb{R}^n$. We must show that $F(x_k) \to F(x_0)$ as $k \to \infty$.

To simplify notation, let the integrand be defined by
\[ h(x, r) = \frac{\Psi(x+r) + \Psi(x-r) - 2\Psi(x)}{|r|^{n+2s}}. \]
Define the sequence of functions $f_k(r) = h(x_k, r)$. From the absolute convergence established earlier, each $f_k(r)$ is Lebesgue-integrable. Furthermore, since $\Psi$ is continuous, we have pointwise convergence:
\[ \lim_{k \to \infty} f_k(r) = h(x_0, r)\eqqcolon f_0(r) \quad \text{for all } r \neq 0. \]

To complete the proof via the DCT, it remains to construct an integrable dominating function $g(r)$ such that $|f_k(r)| \le g(r)$ a.e. for all $k$. Since $\{x_k\}_{k=1}^\infty$ converges, it is bounded and contained in a compact set $K$. We establish bounds uniform for all $x \in K$. To facilitate the analysis, we split the integration region into a local region, $|r| < 1$, and a distant region, $|r|\geq 1$.

\textbf{Local region, $|r| < 1$:} The key idea is to leverage the smoothness of $\Psi(x)$ and rewrite the integrand using a Taylor expansion of $\Psi(x\pm r)$ about $x$ up to first order, with the remainder in Lagrange form. 

Fix $x,r\in\R^n$. Define $g:[-1,1]\to\R$ by $g(t)=\Psi(x+tr)$. Let $\delta(x,r):=\Psi(x+r)+\Psi(x-r)-2\Psi(x)$. Notice that
\begin{equation}
\delta(x,r)=g(1)+g(-1)-2g(0).    
\end{equation}
We need to calculate derivatives at points within a distance of one from the set $K$. Let this set be denoted by $K'$, defined as
\begin{equation*}
    K' = \{ y \in \mathbb{R}^n : \text{dist}(y, K) \le 1 \}.
\end{equation*}

The first order Taylor expansion of $g(t)$ about zero is
\begin{align}\label{psi_taylor_exp}
    g(t)&=g(0)+g'(0)t+\frac{g''(\tau)}{2}t^2,\qquad \tau\in(0,t),
\end{align}
where 
\begin{equation}\label{psi_second_derivative}
    g''(t)=r^T\Hess\Psi(x+tr)r.
\end{equation}
Using equations (\ref{psi_taylor_exp}) and (\ref{psi_second_derivative}), $\delta(x,r)$ can be rewritten as
\begin{align*}
    \delta(x,r)&=\frac{1}{2}g''(\tau_1)+\frac{1}{2}g''(\tau_2)\\
    &=\frac{1}{2} r^T\Hess\Psi(x+\tau_1r)r+\frac{1}{2}r^T\Hess\Psi(x+\tau_2r)r,\qquad \tau_1\in(0,1),\quad \tau_2\in(-1,0)\\
    &=\frac{1}{2} r^T\Hess\Psi(\xi_1)r+\frac{1}{2}r^T\Hess\Psi(\xi_2)r,
\end{align*}
where $\xi_1=x+\tau_1r$, $\xi_2=x+\tau_2r$, and $\xi_1,\xi_2\in B(x,|r|)$, the ball centred at $x$ of radius $|r|$. 

We now calculate a uniform upper bound for $\delta(x,r)$:
\begin{align*}
|\delta(x,r)|&\leq  \frac{1}{2} \left|r^T\Hess\Psi(\xi_1)r\right|+\frac{1}{2}\left|r^T\Hess\Psi(\xi_2)r\right|\\
&=\frac{1}{2} \left|r\cdot\Hess\Psi(\xi_1)r\right|+\frac{1}{2}\left|r\cdot\Hess\Psi(\xi_2)r\right|\\
&\leq\frac{1}{2}|r|\left\|\Hess\Psi(\xi_1)r\right\|+\frac{1}{2}|r|\left\|\Hess\Psi(\xi_2)r\right\|\\
&\leq\frac{1}{2} \left\|\Hess\Psi(\xi_1)\right\|_{op}|r|^2+\frac{1}{2}\left\|\Hess\Psi(\xi_2)\right\|_{op}|r|^2\\
&\leq \frac{1}{2} |r|^2\sup_{\xi\in B(x,|r|)}\left\|\Hess\Psi(\xi)\right\|_{op}+\frac{1}{2}|r|^2\sup_{\xi\in B(x,|r|)}\left\|\Hess\Psi(\xi)\right\|_{op}\\
&=|r|^2\sup_{\xi\in B(x,|r|)}\left\|\Hess\Psi(\xi)\right\|_{op},
\end{align*}
where 
\begin{equation*}
    \|\Hess\Psi(x)\|_{op}=\sup_{|v|=1}|\Hess\Psi(x)v|.
\end{equation*}
Since $\Psi(x)$ is smooth, the Hessian map $x\mapsto\Hess\Psi(x)$ is continuous. Moreover, the norm map $A\mapsto \|A\|_{op}$ is also continuous. The Hessian norm map $x\mapsto\|\Hess\Psi(x)\|_{op}$ is a composition of continuous maps, so it, too, is continuous. It follows, from the Extreme Value Theorem, that $\|\Hess\Psi(x)\|_{op}$ attains an absolute maximum value on the compact set $K'$:
\begin{equation}
    M_1=\sup_{x\in K'}\|\Hess\Psi(x)\|_{op}<\infty.
\end{equation}
As $|r|<1$ and $x\in K$, any ball $B(x,|r|)$ is a proper subset of $K'$. Thus
\begin{equation*}
    |\delta(x,r)|\leq |r|^2\sup_{\xi\in B(x,|r|)}\|\Hess\Psi(\xi)\|_{op}\leq M_1|r|^2.
\end{equation*}
In particular, 
\begin{equation}
    |\delta(x_k,r)|\leq M_1|r|^2
\end{equation}
for every $k$ and $|r|<1$.

Thus, the integrand $h(x_k,r)$ is bounded above by
\begin{equation*}
|h(x_k,r)|\leq \frac{M_1|r|^2}{|r|^{n+2s}}=M_1|r|^{2-2s-n}.
\end{equation*}
A dominating function for the local region is therefore
\[
g_1(r)=M_1|r|^{2-2s-n}.
\]
We verify this function is integrable. Transforming to general spherical coordinates, we have $r=t\omega$, where $t=|r|$ and $\omega\in S^{n-1}$, $S^{n-1}$ denoting the unit sphere in $\R^n$. Let $\omega_n$ be the surface area of $S^{n-1}$. Then,
\begin{align*}
    \int_{B(0,1)}g_1(r)dr&=M_1\int_{B(0,1)}|r|^{2-2s-n}dr\\
    &=\int_{S^{n-1}}\int_{0}^1 t^{2-2s-n}\cdot t^{n-1}dtdS^{n-1}(\omega)\\
    &=\omega_n M_1\int_{0}^1 t^{1-2s}{\,dt}\\
    &<\infty
\end{align*}
provided $2-2s>0$ or $s<1$. Since we assume $s\in(0,1)$, this condition is satisfied.

\textbf{Distant region, $|r|\geq 1$:} First note that for $\alpha >0$ and every $y\in\R^n$,
\begin{equation}\label{psi_global_growth}
    |\Psi(y)|\leq C'(1+|y|^\alpha),
\end{equation}
as when $|y|<1$, $|\Psi(y)|\leq 2C$, and when $|y|\geq 1$, $|\Psi(y)|\leq 2C|y|^\alpha$. 

Since $\{x_k\}$ is bounded, there exists $D>0$ such that $|x|\le D$ for all $x\in K$. Hence, for $|r|\ge 1$,
\begin{equation}\label{xpmbound}
    |x\pm r|\le |x|+|r|\le D+|r|\le (D+1)|r|.
\end{equation}
Combining \eqref{psi_global_growth} and \eqref{xpmbound} yields, for $|r|\ge 1$ and $x\in K$,
\[
|\Psi(x\pm r)|
\le C'(1+|x\pm r|^\alpha)
\le C'(1+(D+1)^\alpha|r|^\alpha)
\le E_1|r|^\alpha,
\]
where in the last step we used $|r|^\alpha\ge 1$ (since $\alpha>0$) and absorbed constants into $E_1>0$.

Moreover, by continuity of $\Psi$ and compactness of $K$,
\[
E_0:=\sup_{x\in K}|\Psi(x)|<\infty.
\]
Therefore, for $|r|\ge 1$ and $x\in K$,
\[
|\delta(x,r)|
\le |\Psi(x+r)|+|\Psi(x-r)|+2|\Psi(x)|
\le 2E_1|r|^\alpha+2E_0
\le H|r|^\alpha,
\]
again using $|r|^\alpha\ge 1$ to absorb the constant term. Consequently,
\[
|h(x_k,r)|\le \frac{H|r|^\alpha}{|r|^{n+2s}}=H|r|^{\alpha-2s-n}
\qquad\text{for } |r|\ge 1.
\]
Thus we may take $g_2(r):=H|r|^{\alpha-2s-n}$ on $\{|r|\ge 1\}$.

Checking integrability, by transforming to general spherical coordinates as done for the local region, we have
\begin{align*}
    \int_{|r|\geq 1}g_2(r)dr &= H\int_{|r|\geq 1}|r|^{\alpha-2s-n}dr\\
    &=\omega_n H\int_1^\infty t^{\alpha-2s-n}\cdot t^{n-1}dt\\
    &=\omega_n H\int_1^\infty t^{\alpha-2s-1}dt\\
    &<\infty
\end{align*}
provided $\alpha-2s<0$ or $\alpha <2s$. Since we assume that $\alpha\in (0,2s)$, for $s\in(0,1)$, this condition is satisfied.

\textbf{Dominating function:} Let $g(r) = g_1(r)\chi_{|r|<1} + g_2(r)\chi_{|r|\ge 1}$. We have demonstrated that $g(r) \in L^1(\R^n)$ and $|f_k(r)|=|h(x_k, r)| \le g(r)$ a.e for every $k$. By the DCT, $\lim_{k\to\infty} F(x_k) = F(x_0)$. Thus, $F(x)$ is continuous on $\R^n$.\\

\textbf{Step 2: Convergence to zero at infinity.} The goal is to show that $\lim_{|x|\to\infty} F(x) = 0$. Let $R = |x|$, and assume $R>2$. We split the domain of integration into the near field ($|r| \le R/2$) and the far field ($|r| > R/2$):
\begin{equation}
|F(x)| \le \frac{C_{n,s}}{2} \left( \underbrace{\int_{|r|\le R/2} \frac{|\delta(x,r)|}{|r|^{n+2s}} dr}_{I_{N}} + \underbrace{\int_{|r|> R/2} \frac{|\delta(x,r)|}{|r|^{n+2s}} dr}_{I_{F}} \right).
\end{equation}

\textbf{Near field, $|r|\leq R/2$:} From the derivation of the Lagrange error bound from the local-region analysis in Step 1, we have
\[
|\delta(x,r)|\leq |r|^2\sup_{\xi\in B(x,|r|)}\left\|\Hess\Psi(\xi)\right\|_{op}.
\] 
Differentiating $\Psi(x)$ twice gives
\begin{equation}
    D^2\Psi(x)=\alpha(1+|x|^2)^{\alpha/2-1}I+\alpha(\alpha-2)(1+|x|^2)^{\alpha/2-2}(x\otimes x),
\end{equation}
where $I$ denotes the identity matrix and $x\otimes x$ is the outer product. To calculate an upper bound for the operator norm of $D^2\Psi(x)$, we first analyse $\|I\|_{op}$ and $\|x\otimes x\|_{op}$.

The operator norm of $I$ is
\begin{equation}
    \|I\|_{op}=\sup_{|v|=1}|Iv|=\sup_{|v|=1}|v|=1.    
\end{equation}

Writing the outer product as
\[
x\otimes x=x x^\top = 
\begin{pmatrix} 
x_1x_1 & x_1 x_2 & \dots & x_1 x_n \\ 
x_2 x_1 & x_2x_2 & \dots & x_2 x_n \\ 
\vdots & \vdots & \ddots & \vdots \\ 
x_n x_1 & x_n x_2 & \dots & x_nx_n 
\end{pmatrix}=\begin{pmatrix} 
x_1x & x_2x & \dots & x_nx
\end{pmatrix}
\]
shows that it is a rank-1 symmetric matrix, since every column is a multiple of the column vector $x$, and $x_ix_j=x_jx_i$. It follows that $x\otimes x$ has $n$ real eigenvalues, which we denote by $\lambda_1,\lambda_2,\ldots,\mbox{and }\lambda_n$, and the spectral theorem applied to real symmetric matrices yields
\begin{equation}
    \|x\otimes x\|_{op}=\max_{i}\{\lambda_1,\lambda_2,\ldots,\lambda_n\}.
\end{equation}
Notice that as
\[
(x\otimes x)x=(xx^\top)x=x(x^\top x)=|x|^2x,
\]
$x$ is an eigenvector of $x\otimes x$ with eigenvalue $|x|^2$. Moreover, any vector $v$ that is orthogonal to $x$ is both in the kernel of $x\otimes x$ and an eigenvector with eigenvalue $0$, since
\[
(x\otimes x)v=x(x^\top v)=0x=0.
\]
As $x\otimes x$ has rank $1$, the rank-nullity theorem implies that the dimension of its kernel is $n-1$, which is also the dimension of the $0$-eigenspace. Thus the $|x|^2$-eigenspace has dimension 1, and we know all of the eigenvalues of $x\otimes x$. So
\begin{equation*}
    \|x\otimes x\|_{op}=|x|^2.
\end{equation*}

The operator norm of $\Hess\Psi(x)$ is therefore bounded above by
\begin{align*}
    \|\Hess\Psi(x)\|_{op}
    &=\left\|\alpha(1+|x|^2)^{\alpha/2-1}I+\alpha(\alpha-2)(1+|x|^2)^{\alpha/2-2}(x\otimes x)\right\|_{op}\\
    &\leq \alpha(1+|x|^2)^{\alpha/2-1}\|I\|_{op}+|\alpha(\alpha-2)|(1+|x|^2)^{\alpha/2-2}\|x\otimes x\|_{op}\\
    &=\alpha(1+|x|^2)^{\alpha/2-1}+|\alpha(\alpha-2)|(1+|x|^2)^{\alpha/2-2}|x|^2.
\end{align*}
Let the first term be denoted by $T_1$ and the second term be denoted by $T_2$. When $|x|\geq 1$ (which is satisfied as we assume $R:=|x|>2$), we have
\[
T_1=\alpha(1+|x|^2)^{\alpha/2-1}\leq \alpha |x|^{\alpha-2}\leq F_1|x|^{\alpha-2}.
\]
Similarly,
\[
T_2=|\alpha(\alpha-2)|(1+|x|^2)^{\alpha/2-2}|x|^2\leq F_2|x|^{\alpha-2}.
\]
It follows that
\begin{equation}\label{upper_bound_hess_2}
    \|\Hess\Psi(x)\|_{op}\leq F_1|x|^{\alpha-2}+F_2|x|^{\alpha-2}\leq F|x|^{\alpha-2},\quad |x|\geq 1.
\end{equation}

On the basis of the foregoing analysis, we can now calculate an upper bound for $|\delta(x,r)|$ in terms of $R$ and $r$. Observe that any point $\xi\in B(x,|r|)$ satisfies 
\[
|\xi|\geq |x|-|r|\geq R-R/2=R/2>1,
\]
since we assume that $|r|\leq R/2$. So using the upper bound from inequality (\ref{upper_bound_hess_2}), we have
\[
\|\Hess\Psi(\xi)\|_{op}\leq F|\xi|^{\alpha-2},\quad |\xi|\geq 1.
\]
Since $0<\alpha<2s<2$ for $0<s<1$, the power function decays as $|\xi|\to\infty$. Thus the upper bound is maximised when $|\xi|=R/2$. So
\[
\sup_{\xi\in B(x,|r|)}\|\Hess\Psi(\xi)\|_{op}\leq J(R/2)^{\alpha-2}, 
\]
giving
\[
|\delta(x,r)|\leq|r|^2\sup_{\xi\in B(x,|r|)}\|\Hess\Psi(\xi)\|_{op}\leq J(R/2)^{\alpha-2}|r|^2.
\]

We can now determine the limiting behaviour of the near-field integral $I_{N}$:
\begin{align*}
    I_N&\leq\int_{|r|\leq R/2}\frac{|\delta(x,r)|}{|r|^{n+2s}}{\,dr}\\
    &\leq J(R/2)^{\alpha-2}\int_{|r|\leq R/2}|r|^{2-2s-n}dr\\
    &=J'R^{\alpha-2}\int_0^{R/2}t^{2-2s-n}t^{n-1}dt\\
    &=J'R^{\alpha-2}\int_0^{R/2}t^{1-2s}dt\\
    &=\frac{J'R^{\alpha-2}}{2-2s}\left[t^{2-2s}\right]_0^{R/2},\quad s\in(0,1)\\
    &=\frac{J'R^{\alpha-2}}{2-2s}\left[(R/2)^{2-2s}-0\right]\\
    &=G|x|^{\alpha-2s},\quad R:=|x|,
\end{align*}
so $0\leq I_N \leq G|x|^{\alpha-2s}\to 0$ as $|x|\to \infty$, since $\alpha<2s$. Thus, $\lim_{|x|\to\infty}I_N=0$.

\textbf{Far field, $|r|>R/2$:} From the continuity proof in the distant region in Step 2, we derived
\[
|\Psi(y)|\leq C'(1+|y|^\alpha)
\]
for any $\alpha\in(0,2s)$ and $s\in(0,1)$. Thus,
\begin{equation}\label{dist_psi_ubound}
  |\Psi(x\pm r)|\leq C'(1+|x\pm r|^\alpha)\leq C'(1+(|x|+|r|)^\alpha), \quad \alpha > 0,  
\end{equation}
assuming $x,r\in\R^n$ such that $|x|=:R>2$ and $|r|>R/2>1$. 

We now analyse the power function $f(t)=t^p$ when $p>0$ and $t\geq 0$. Notice that when $p\in(0,1]$, the function is concave on $(0,\infty)$ and $f(0)=0$. The function $f$ is therefore subadditive, so 
\[
a^p+b^p\geq (a+b)^p.
\]
When $p>1$, the function $t^p$ is strictly convex on $(0,\infty)$, and so we may apply Jensen's inequality to give
\[
(a+b)^p\leq 2^{p-1}(a^p+b^p).
\]
Combining these two inequalities gives
\begin{equation}\label{inequa_power_sums}
    (a+b)^p\leq C_p(a^p+b^p),    
\end{equation}
for $p>0$ and $a,b\in[0,\infty)$. 

Applying the power-of-sums inequality (\ref{inequa_power_sums}) to inequality (\ref{dist_psi_ubound}) yields
\begin{equation}\label{psi_power_sums}
    |\Psi(x\pm r)|\leq C'(1+(|x|+|r|)^\alpha)\leq C'(1+C_p(|x|^\alpha+|r|^\alpha)).    
\end{equation}

An upper bound for $|\delta(x,r)|$ is thus
\begin{align*}
    |\delta(x,r)|\leq |\Psi(x+r)|+|\Psi(x-r)|+2|\Psi(x)|\leq K_1(1+R^\alpha +|r|^\alpha),\quad R:=|x|>2,
\end{align*}
using inequalities (\ref{psi_power_sums}) and (\ref{psi_global_growth}). We may write the far field integral as
\[
I_F = \int_{|r|> R/2} \frac{|\delta(x,r)|}{|r|^{n+2s}} dr \leq K_1(1+R^\alpha)\int_{|r|>R/2}\frac{dr}{|r|^{n+2s}}+K_1\int_{|r|>R/2}\frac{|r|^\alpha}{|r|^{n+2s}}{\,dr}.
\]
The first integral evaluates to
\[
\int_{|r|>R/2}\frac{dr}{|r|^{n+2s}}=\omega_n\int_{R/2}^\infty t^{-n-2s}t^{n-1}dt=\omega_n\int_{R/2}^\infty t^{-1-2s}dt=\frac{\omega_n}{-2s}\left[t^{-2s}\right]_{R/2}^\infty=L_1|x|^{-2s}.
\]
The second integral gives
\[
\int_{|r|>R/2}\frac{|r|^\alpha}{|r|^{n+2s}}{\,dr}=\omega_n\int_{R/2}^\infty t^{\alpha-2s-1}dt=\frac{\omega_n}{\alpha-2s}\left[t^{\alpha-2s}\right]_{R/2}^\infty=L_2|x|^{\alpha-2s}
\]
since $0<\alpha<2s$. So we have
\begin{align*}
    I_F&\leq K_1(1+R^\alpha)\int_{|r|>R/2}\frac{dr}{|r|^{n+2s}}+K_1\int_{|r|>R/2}\frac{|r|^\alpha}{|r|^{n+2s}}{\,dr}\\
    &=K_1(1+R^\alpha)L_1|x|^{-2s}+K_1 L_2|x|^{\alpha-2s}\\
    &=K_1L_1|x|^{-2s}+K_1L_1|x|^{\alpha-2s}+K_1L_2|x|^{\alpha-2s}.
\end{align*}
Each of the terms converges to zero since $s>0$ and $\alpha<2s$. Thus, $\lim_{|x|\to\infty}I_F=0$. 

As $|F(x)|$ converges to zero, it follows that $F(x)$ converges to zero as $|x|\to\infty$. This completes Step 2.
\\

\textbf{Step 3: Global boundedness.} Since $F(x)$ is continuous on $\R^n$ and converges to zero as $|x|\to\infty$, it is globally bounded.\\

\textbf{Step 4: Non-negativity of the global bound.} Define
\[
K_{\Psi}:=\sup_{x\in\R^n}F(x).
\]
By Step 3, this quantity is finite. It suffices to show that $K_{\Psi}$ is non-negative. Notice that $\Psi(z)=(1+|z|^2)^{\alpha/2}\geq 1$ for all $z\in\R^n$ and $\Psi(0)=1$. The singular fractional kernel $K_s(r):=C_{n,s}|r|^{-n-2s}$ is also strictly positive for $r\neq 0$. Thus,
\[
g(r):=\frac{1}{2}K_s(r)(\Psi(0+r)+\Psi(0-r)-2\Psi(0))=K_s(r)(\Psi(r)-1)\geq 0,
\]
with strict inequality for $r\neq 0$. Thus,
\[
F(0)=\int_{\R^n}g(r)dr\geq 0.
\]
It follows that $K_{\Psi}=\sup_{x\in\R^n}F(x)\geq F(0)\geq 0$, so the global bound is non-negative.

\begin{rem}
    The global bound is actually positive, since $g(r)>0$ a.e, so $\int_{\R^n}g(r)dr\neq 0$.
\end{rem}
    
\end{proof}

\section{Persistence of the regional modulus defect in higher dimensions}
\label{app:higher_dimensional_defect}

\begin{proof}[Proof of Proposition~\ref{prop:nd_counterexample}]
Fix $n\geq2$.  We construct a thin convex cylinder and an admissible
comparison function satisfying all the hypotheses of
Conjecture~\ref{thm_nd_mod_cont_reg}, for which the asserted modulus of
continuity fails at positive time.  The construction follows a similar
four-step structure as in the one-dimensional case. Due to significant technical overlap, 
only the new ingredients needed to pass from the interval to the proposed $n$-dimensional convex set are expounded in detail.

\medskip

\noindent
\textbf{Step 1: Specify the domain and kernel.}
Fix $n\geq2$, $0<\epsilon<3/5$, and set
\begin{equation}\label{eq:thin_cylinder_a_eps}
    a_\epsilon:=\sqrt{1-\epsilon^2}.
\end{equation}
Define the closed thin convex cylinder
\begin{equation}\label{eq:thin_cylinder_domain}
    \Omega_\epsilon
    :=
    [-a_\epsilon,a_\epsilon]\times \overline{B_\epsilon^{n-1}},
    \qquad
    \overline{B_\epsilon^{n-1}}
    :=
    \{p\in\R^{n-1}:|p|\leq\epsilon\}.
\end{equation}
Then
\[
    D:=\operatorname{diam}(\Omega_\epsilon)
    =
    \sqrt{(2a_\epsilon)^2+(2\epsilon)^2}
    =2,
\]
so the comparison interval in
Conjecture~\ref{thm_nd_mod_cont_reg} is again $I=[-1,1]$.

The higher-dimensional analogue of the one-dimensional box kernel
\eqref{eq:regional_box_kernel} is
\begin{equation}\label{eq:nd_box_kernel}
    \rho_0(X):=\chi_{B_{1/2}^n}(X),
    \qquad X\in\R^n.
\end{equation}
For the cylinder~\eqref{eq:thin_cylinder_domain}, the $q$th transverse slice is given by
\[
    (\Omega_\epsilon)_q
    =
    \begin{cases}
        \overline{B_\epsilon^{n-1}}, & |q|\leq a_\epsilon,\\
        \varnothing, & |q|>a_\epsilon.
    \end{cases}
\]
Since $0<\epsilon<3/5$, we have
\[
    a_\epsilon>\frac45>\frac12.
\]
If $|q|>a_\epsilon$, then
$|(q,p)|\geq|q|>1/2$ for every $p\in\R^{n-1}$, so
$\rho_0(q,p)=0$.  Thus the $q$-slices outside
$[-a_\epsilon,a_\epsilon]$ make no contribution to the induced kernel,
and we may write
\begin{equation}\label{eq:nd_induced_kernel_cylinder}
    \tilde\rho_\epsilon(q)
    = \int_{\overline{B_\epsilon^{n-1}}}\rho_0(q,p)\,dp,\qquad q\in\R.
\end{equation}
For $|q|<1/2$, apply~\eqref{eq:nd_box_kernel} to obtain
\[
    \tilde\rho_\epsilon(q)
    =
    |B^{n-1}_1|
    \min\left\{
        \epsilon,\sqrt{\frac14-q^2}
    \right\}^{n-1}.
\]
Therefore
\begin{equation}
    \tilde\rho_\epsilon(q)
    =
    \begin{cases}
    \displaystyle
    |B^{n-1}_1|
    \min\left\{
        \epsilon,\sqrt{\frac14-q^2}
    \right\}^{n-1},
        & |q|<\frac12,\\[2mm]
    0, & |q|\geq\frac12.
    \end{cases}
    \label{eq:nd_induced_kernel_explicit}
\end{equation}

Set
\begin{equation}\label{eq:nd_ceps_keps}
    c_\epsilon
    :=
    |B^{n-1}_1|\epsilon^{n-1},
    \qquad
    k_\epsilon(q)
    :=
    \frac{\tilde\rho_\epsilon(q)}{c_\epsilon}.
\end{equation}
Then
\[
    0\leq k_\epsilon(q)
    \leq
    \chi_{[-1/2,1/2]}(q).
\]
Also, for each fixed $|q|<1/2$, we may choose $\epsilon$ sufficiently small such that
\[
    \epsilon<\sqrt{\frac14-q^2},
\]
giving $k_\epsilon(q)=1$.  The dominated convergence theorem then yields
\begin{equation}\label{eq:keps_L1_convergence}
    k_\epsilon
    \longrightarrow
    k_0:=\chi_{[-1/2,1/2]}
    \qquad\mbox{in }L^1(\R)
    \quad\text{as }\epsilon\to0.
\end{equation}
So after normalisation by the transverse volume $c_\epsilon$, the
induced kernel converges to the one-dimensional box kernel
used previously.  This thin-domain reduction is the principal new geometric
ingredient in the higher-dimensional construction.

As in the one-dimensional case, the box kernel is used only as a
reference kernel.  Let $\gamma_\eta$ be a standard non-negative,
radially symmetric and radially non-increasing mollifier in $\R^n$,
supported in the closed ball $\overline{B_\eta^n}$, and define
\begin{equation}\label{eq:nd_mollified_kernel}
    \rho_\eta:=\rho_0*\gamma_\eta.
\end{equation}
By the standard properties of mollification, $\rho_\eta$ is smooth,
non-negative and radial.  Moreover, as $\operatorname{supp}\rho_0\subset\overline{B_{1/2}^n}$ and $\operatorname{supp}\gamma_\eta\subset\overline{B_\eta^n}$, it follows that
\begin{equation}\label{eq:nd_mollified_support}
    \operatorname{supp}\rho_\eta
    \subset
    \overline{B_{1/2+\eta}^n}.
\end{equation}
Hence $\rho_\eta$ is compactly supported and trivially decays to zero
at infinity.  As both $\rho_0$ and $\gamma_\eta$ are non-negative,
radial and radially non-increasing, their convolution $\rho_\eta$ is
also radially non-increasing; see, for example, Cañizo's online proof
of this standard convolution property.\footnote{J.~A. Cañizo,
\emph{Convolution of symmetric, radially decreasing functions},
\url{https://canizo.org/convolution/}.}
The higher-dimensional analogue of the transition-region
argument used in Step~1 of the one-dimensional construction additinally gives
\begin{equation}\label{eq:nd_mollifier_L1}
    \|\rho_\eta-\rho_0\|_{L^1(\R^n)}
    \longrightarrow0
    \qquad\text{as }\eta\to0.
\end{equation}

We have now defined a suitable domain and kernel.

\medskip

\noindent
\textbf{Step 2: Construct a solution $u_{\epsilon,\eta}$ with suitable
properties.}
Choose the linear initial datum
\begin{equation}\label{eq:nd_linear_initial_data}
    F(x_1,x'):=\frac{x_1}{2},
    \qquad (x_1,x')\in\Omega_\epsilon,
\end{equation}
and let $u_{\epsilon,\eta}$ solve
\begin{equation}\label{eq:nd_mollified_evolution}
    (u_{\epsilon,\eta})_t
    =
    L_{\Omega_\epsilon}^{\rho_\eta}u_{\epsilon,\eta},
    \qquad
    u_{\epsilon,\eta}(\cdot,0)=F.
\end{equation}
Since
\[
    \|F\|_{L^\infty(\Omega_\epsilon)}
    =
    \frac{a_\epsilon}{2}
    <
    \frac12,
\]
and $\Omega_\epsilon$ is compact, the continuous function
$u_{\epsilon,\eta}(\cdot,t)$ attains its spatial maximum and minimum on
$\Omega_\epsilon$ for each $t\geq0$.  The same maximum-principle
argument used in Step~2 of the one-dimensional counterexample therefore
gives
\begin{equation}\label{eq:nd_solution_bound}
    \|u_{\epsilon,\eta}(\cdot,t)\|_{L^\infty(\Omega_\epsilon)}
    \leq
    \frac{a_\epsilon}{2}
    <
    1
\end{equation}
for every $t\geq0$.

\medskip
\noindent
\textbf{Step 3: Construct an appropriate comparison function
$\varphi_{\epsilon,\eta}$.}
We use the induced one-dimensional kernel. Let $0<\eta<a_\epsilon-\frac12$, so that $a_\epsilon>\eta+1/2$.
As the support of $\rho_\eta$ is contained in $\overline{B_{1/2+\eta}^n}$, its projection onto the $q$-axis lies inside $[-a_\epsilon,a_\epsilon]$, where the transverse slices of the cylinder $(\Omega_\epsilon)_q$ entering the induced kernel are the constant balls $\overline{B_\epsilon^{n-1}}$. The induced kernel is therefore given by

\begin{equation}\label{eq:nd_induced_smooth_kernel}
    \tilde\rho_{\epsilon,\eta}(q)
    =
    \int_{\overline{B_\epsilon^{n-1}}}
        \rho_\eta(q,p)\,dp,
    \qquad q\in\R.
\end{equation}

Since $\rho_\eta$ is smooth, non-negative, radial and compactly
supported, $\tilde\rho_{\epsilon,\eta}$ is smooth, non-negative, even
and compactly supported.  If $0\leq q_1\leq q_2$, then for
every $p\in\overline{B_\epsilon^{n-1}}$,
\[
    \sqrt{q_1^2+|p|^2}
    \leq
    \sqrt{q_2^2+|p|^2}.
\]
As $\rho_\eta$ is non-increasing as a function of distance,
\[
    \rho_\eta(q_1,p)
    \geq
    \rho_\eta(q_2,p).
\]
Integrating over $\overline{B_\epsilon^{n-1}}$ gives
\[
    \tilde\rho_{\epsilon,\eta}(q_1)
    \geq
    \tilde\rho_{\epsilon,\eta}(q_2).
\]
Thus $\tilde\rho_{\epsilon,\eta}$ is non-increasing as a function of
$|q|$ and has the same structural properties as the one-dimensional
kernels considered above. Moreover, observe that
\[
\|\tilde\rho_{\epsilon,\eta}-\tilde\rho_{\epsilon,0}\|_{L^1(\R)}\leq \|\rho_\eta-\rho_0\|_{L^1(\R^n)}\to 0.
\]

Let $\tilde\varphi_{\epsilon,\eta}$ solve the one-dimensional regional
heat equation on $I=[-1,1]$ with kernel
$\tilde\rho_{\epsilon,\eta}$ and linear initial datum:
\begin{equation}\label{eq:nd_exact_comparison_choice}
\begin{cases}
\displaystyle
(\tilde\varphi_{\epsilon,\eta})_t(r,t)
=
\int_{-1}^{1}
    \tilde\rho_{\epsilon,\eta}(w-r)
    \bigl(
        \tilde\varphi_{\epsilon,\eta}(w,t)
        -
        \tilde\varphi_{\epsilon,\eta}(r,t)
    \bigr)\,dw,
    \\[2.2ex]
\displaystyle
\tilde\varphi_{\epsilon,\eta}(r,0)=\frac r2,
\end{cases}
\qquad r\in[-1,1].
\end{equation}
Define
\begin{equation}\label{eq:nd_phi_definition}
    \varphi_{\epsilon,\eta}(r,t)
    :=
    \tilde\varphi_{\epsilon,\eta}(r,t),
    \qquad 0\leq r\leq1.
\end{equation}

We now verify properties~(a.)--(c.) of
Conjecture~\ref{thm_nd_mod_cont_reg}.

\smallskip

\noindent
\emph{Property (a.).}
The kernel $\tilde\rho_{\epsilon,\eta}$ is even and the initial datum
$r/2$ is odd. Using the same oddness-preservation argument as
in Step~2 of the one-dimensional counterexample, we conclude that
\begin{equation}\label{eq:nd_phi_oddness}
    \tilde\varphi_{\epsilon,\eta}(-r,t)
    =
    -\tilde\varphi_{\epsilon,\eta}(r,t), 
\end{equation}
and in particular,
\[
    \varphi_{\epsilon,\eta}(0,t)=0.
\]
Furthermore, because $\tilde\rho_{\epsilon,\eta}$ is even and
non-increasing as a function of distance, the strict-monotonicity
argument from Step~2 of the one-dimensional construction applies
without change to~\eqref{eq:nd_exact_comparison_choice}. Therefore,
\begin{equation}\label{eq:nd_phi_monotonicity}
    (\varphi_{\epsilon,\eta})_r(r,t)>0,
    \qquad 0\leq r\leq1,\quad t\geq0.
\end{equation}
Continuity of $\tilde\varphi_{\epsilon,\eta}$ gives
\[
    \lim_{r\to0^+}\varphi_{\epsilon,\eta}(r,t)
    =
    \varphi_{\epsilon,\eta}(0,t)
    =
    0.
\]
Thus property~(a.) holds.  Additionally,
\[
    \varphi_{\epsilon,\eta}(r,t)>0
    \qquad\text{for }r>0.
\]

\smallskip

\noindent
\emph{Property (b.).} For any two points
\[
    P=(p_1,p'),
    \qquad
    Q=(q_1,q')
    \in\Omega_\epsilon,
\]
the initial data give
\begin{align}
    |u_{\epsilon,\eta}(Q,0)-u_{\epsilon,\eta}(P,0)|
    &=
    \frac{|q_1-p_1|}{2}\leq
    \frac{|Q-P|}{2}=
    2\varphi_{\epsilon,\eta}
       \left(\frac{|Q-P|}{2},0\right).
    \label{eq:nd_initial_modulus}
\end{align}
Since $\operatorname{diam}(\Omega_\epsilon)=2$, $0\leq\frac{|Q-P|}{2}\leq 1$, which 
ensures the argument of $\varphi_{\epsilon,\eta}$ lies in its domain.
It follows that $\varphi_{\epsilon,\eta}(\cdot,0)$ is a modulus of continuity for
$u_{\epsilon,\eta}(\cdot,0)$, and property~(b.) holds.

\smallskip

\noindent
\emph{Property (c.).}
By~\eqref{eq:nd_phi_oddness},
$\tilde\varphi_{\epsilon,\eta}$ is the odd extension to
$I=[-1,1]$ of its restriction
$\varphi_{\epsilon,\eta}$ to $[0,1]$.  Therefore, for $0<r\leq1$,
\begin{align}
    (\varphi_{\epsilon,\eta})_t(r,t)
    &=
    (\tilde\varphi_{\epsilon,\eta})_t(r,t)=
    \int_{-1}^{1}
        \tilde\rho_{\epsilon,\eta}(w-r)
        \bigl(
            \tilde\varphi_{\epsilon,\eta}(w,t)
            -
            \tilde\varphi_{\epsilon,\eta}(r,t)
        \bigr)\,dw.
    \label{eq:nd_phi_exact_supersolution}
\end{align}
Thus the supersolution condition in property~(c.) holds with equality.

Consequently, $\varphi_{\epsilon,\eta}$ satisfies all the hypotheses
imposed on the comparison function in
Conjecture~\ref{thm_nd_mod_cont_reg}.

\medskip

\noindent
\textbf{Step 4: Show the conjecture fails.}
We retain the same longitudinal coordinates as in the one-dimensional case:
\begin{equation}\label{eq:nd_counterexample_points}
    X:=\left(-\frac45,0\right),
    \qquad
    Y:=\left(\frac35,0\right).
\end{equation}
Since $\epsilon<3/5$ implies $a_\epsilon>4/5$, both points belong to
$\Omega_\epsilon$, and
\begin{equation}\label{eq:nd_counterexample_half_distance}
    \frac{|Y-X|}{2}=\frac7{10}.
\end{equation}
Define
\begin{equation}\label{eq:nd_defect}
    G_{\epsilon,\eta}(t)
    :=
    u_{\epsilon,\eta}(Y,t)
    -
    u_{\epsilon,\eta}(X,t)
    -
    2\varphi_{\epsilon,\eta}
       \left(\frac7{10},t\right).
\end{equation}
Equality holds in~\eqref{eq:nd_initial_modulus} for this pair, and hence
\begin{equation}\label{eq:nd_defect_zero}
    G_{\epsilon,\eta}(0)=0.
\end{equation}

First evaluate the initial defect for the reference box kernel
$\rho_0$.  Write $u_{\epsilon,0}$,
$\varphi_{\epsilon,0}$ and $G_{\epsilon,0}$ for the corresponding
functions.  For $P:=(p,0)$, integration in the transverse variables
gives
\begin{align}
    (u_{\epsilon,\eta})_t=L_{\Omega_\epsilon}^{\rho_0}F(P)
    &=
    \int_{-a_\epsilon}^{a_\epsilon}
    \int_{\overline{B_\epsilon^{n-1}}}
        \rho_0(z-p,z')
        \frac{z-p}{2}\,dz'\,dz\nonumber\\
    &=
    \int_{-a_\epsilon}^{a_\epsilon}
        \tilde\rho_\epsilon(z-p)
        \frac{z-p}{2}\,dz\nonumber\\
    &=
    c_\epsilon
    \int_{-a_\epsilon}^{a_\epsilon}
        k_\epsilon(z-p)
        \frac{z-p}{2}\,dz.
    \label{eq:nd_actual_operator_reduced}
\end{align}
Similarly,
\begin{equation}\label{eq:nd_comparison_derivative}
    (\varphi_{\epsilon,0})_t
       \left(\frac7{10},0\right)
    =
    c_\epsilon
    \int_{-1}^{1}
        k_\epsilon\left(w-\frac7{10}\right)
        \frac{w-7/10}{2}\,dw.
\end{equation}
Consequently,
\begin{align}
    \frac{G_{\epsilon,0}'(0)}{c_\epsilon}
    &=
    \int_{-a_\epsilon}^{a_\epsilon}
        k_\epsilon\left(z-\frac35\right)
        \frac{z-3/5}{2}\,dz\nonumber\\
    &\quad-
    \int_{-a_\epsilon}^{a_\epsilon}
        k_\epsilon\left(z+\frac45\right)
        \frac{z+4/5}{2}\,dz\nonumber\\
    &\quad-
    2\int_{-1}^{1}
        k_\epsilon\left(w-\frac7{10}\right)
        \frac{w-7/10}{2}\,dw.
    \label{eq:nd_normalised_defect}
\end{align}
This is the higher-dimensional counterpart of the defect functional
\eqref{eq:regional_defect_functional}.  The only differences are that
the one-dimensional box kernel is replaced by $k_\epsilon$ and that the
two terms arising from the higher-dimensional solution are integrated
over $[-a_\epsilon,a_\epsilon]$ rather than $[-1,1]$.

For $p\in\{-4/5,3/5\}$, let
\[
    A_\epsilon(p)
    :=
    \int_{-a_\epsilon}^{a_\epsilon}
        k_\epsilon(z-p)\frac{z-p}{2}\,dz
\]
and
\[
    A_0(p)
    :=
    \int_{-1}^{1}
        k_0(z-p)\frac{z-p}{2}\,dz
    =
    L_I^{k_0}f(p),
    \qquad f(r):=\frac r2.
\]
As $|z-p|/2\leq1$ for $z,p\in[-1,1]$,
\begin{align}
    |A_\epsilon(p)-A_0(p)|
    &\leq
    \|k_\epsilon-k_0\|_{L^1(\R)}
    +
    2(1-a_\epsilon).
    \label{eq:Aeps_convergence_bound}
\end{align}
Likewise,
\begin{align}
    \left|
    \int_{-1}^{1}
        k_\epsilon\left(w-\frac7{10}\right)
        \frac{w-7/10}{2}\,dw
    -
    L_I^{k_0}f\left(\frac7{10}\right)
    \right|
    &\leq
    \|k_\epsilon-k_0\|_{L^1(\R)}.
    \label{eq:Beps_convergence_bound}
\end{align}
It follows that
\begin{align}
    &\left|
    \frac{G_{\epsilon,0}'(0)}{c_\epsilon}
    -
    \left[
        L_I^{k_0}f\left(\frac35\right)
        -
        L_I^{k_0}f\left(-\frac45\right)
        -
        2L_I^{k_0}f\left(\frac7{10}\right)
    \right]
    \right|\nonumber\\
    &\qquad\leq
    4\|k_\epsilon-k_0\|_{L^1(\R)}
    +
    4(1-a_\epsilon)
    \longrightarrow0.
    \label{eq:nd_defect_convergence_bound}
\end{align}
The expression in square brackets is the one-dimensional
defect already calculated in
\eqref{eq:regional_positive_defect}.  Hence
\begin{equation}\label{eq:nd_defect_limit}
    \frac{G_{\epsilon,0}'(0)}{c_\epsilon}
    \longrightarrow
    \frac1{200}>0
    \qquad\text{as }\epsilon\to0.
\end{equation}
Therefore
\[
    G_{\epsilon,0}'(0)>0
\]
for every sufficiently small $\epsilon>0$.

Fix one such $\epsilon$.  We now pass from the reference box kernel to
the smooth kernel $\rho_\eta$.  By
\eqref{eq:nd_mollifier_L1} and the same kernel-continuity estimate from Step~4 of the one-dimensional construction, we have 
\[
    \left|
        G_{\epsilon,\eta}'(0)
        -
        G_{\epsilon,0}'(0)
    \right|
    \leq
    4\|\rho_\eta-\rho_0\|_{L^1(\R^n)}
    \longrightarrow0.
\]
Thus $G_{\epsilon,\eta}'(0)>0$ for all sufficiently small
$\eta>0$ satisfying $\eta<a_\epsilon-1/2$.

Now fix such an $\eta$.  Combining this with
\eqref{eq:nd_defect_zero}, differentiability at $t=0$ implies
\[
    G_{\epsilon,\eta}(t)>0
\]
for all sufficiently small $t>0$.  As
$\varphi_{\epsilon,\eta}(r,t)\geq0$ for $r\geq0$,
\[
    u_{\epsilon,\eta}(Y,t)
    -
    u_{\epsilon,\eta}(X,t)
    >
    2\varphi_{\epsilon,\eta}
       \left(\frac7{10},t\right)
    \geq0.
\]
Consequently,
\[
    \left|
        u_{\epsilon,\eta}(Y,t)
        -
        u_{\epsilon,\eta}(X,t)
    \right|
    >
    2\varphi_{\epsilon,\eta}
       \left(\frac{|Y-X|}{2},t\right)
\]
for all sufficiently small positive $t$, contradicting the conclusion
of Conjecture~\ref{thm_nd_mod_cont_reg}.
\end{proof}

\section*{Declarations}


\begin{itemize}
\item Funding.
The results in this paper form a core part of the second author's doctoral thesis, completed at the Mathematical Sciences Institute of the Australian National University (awarded April 2025). Her work was supported by an Australian Postgraduate Award. 
\item Conflict of interest/Competing interests. The authors have no competing interests to declare that are relevant to the content of this article.
\item Ethics approval and consent to participate. Not applicable.
\item Consent for publication. Not applicable.
\item Data availability. Not applicable.
\item Materials availability. Not applicable.
\item Code availability. Not applicable.
\item Author contribution. Conceptualisation: B. Andrews; Methodology: S. Chen and B. Andrews; Investigation: S. Chen; Writing: S. Chen; Writing – Review and Editing: S. Chen and B. Andrews; Supervision: B. Andrews. All authors have read and agreed to the published version of the manuscript.
\item Declaration of generative AI and AI-assisted technologies in the manuscript preparation process. The second author used AI-assisted tools,
including Aristotle (Harmonic), ChatGPT, Claude, and Gemini, to assist with aspects of the Lean 4 formalisation
of Theorem~2.3 and its proof, setting up an automated Lean workflow, and exploring
extensions of Theorem~2.3 to operators including the fractional Laplacian. These
tools were also used to review the original statement and proof of Theorem~2.3,
helping to identify an error that was subsequently resolved by the authors with
AI-assisted checking, and to assist with checking and editing the counterexamples and appendixes.

The Lean formalisation repository is currently under review and can be made available on request. 

The authors reviewed and verified all mathematical statements, proofs, citations, and final text, and take full responsibility for the content of the manuscript.
\end{itemize}

\section*{Acknowledgments}
We thank Abraham Ng for his assistance in applying the `coupling-by-reflection' technique in the proof of Theorem \ref{thm_modulus_Rn}.

\bigskip

\address{Mathematical Science Institute, Australian National University, ACT 0200, Australia} \curraddr{}

\email{Email address: ben.andrews@anu.edu.au} \thanks{}\\ 

\address{Independent Researcher, ACT 2606, Australia} \curraddr{}

\email{Email address: sophie.c.chen@gmail.com} \thanks{}

\begin{thebibliography}{99}

\bibitem{AC}
B.~Andrews, J.~Clutterbuck,
Proof of the fundamental gap conjecture,
J. Amer. Math. Soc. 24~(3) (2011) 899--916.

\bibitem{AC1}
B.~Andrews, J.~Clutterbuck,
Time-interior gradient estimates for quasilinear parabolic equations,
Indiana Univ. Math. J. 58~(1) (2009) 351--380.

\bibitem{AC2}
B.~Andrews, J.~Clutterbuck,
Lipschitz bounds for solutions of quasilinear parabolic equations in one space variable,
J. Differential Equations 246~(11) (2009) 4268--4283.

\bibitem{AC3}
B.~Andrews, J.~Clutterbuck,
Sharp modulus of continuity for parabolic equations on manifolds and lower bounds for the first eigenvalue,
Anal. PDE 6~(5) (2013) 1013--1024.

\bibitem{LiFullyNonlinear}
X.~Li,
Modulus of continuity estimates for fully nonlinear parabolic equations,
Calc. Var. Partial Differential Equations 60 (2021) 182.

\bibitem{KNV}
A.~Kiselev, F.~Nazarov, A.~Volberg,
Global well-posedness for the critical 2D dissipative quasi-geostrophic equation,
Invent. Math. 167~(3) (2007) 445--453.

\bibitem{SchillingSztonykWang}
R.~L. Schilling, P.~Sztonyk, J.~Wang,
Coupling property and gradient estimates of L\'evy processes via the symbol,
Bernoulli 18~(4) (2012) 1128--1149.

\bibitem{KulczyckiRyznar}
T.~Kulczycki, M.~Ryznar,
Gradient estimates of harmonic functions and transition densities for L\'evy processes,
Trans. Amer. Math. Soc. 368~(1) (2016) 281--318.

\end{thebibliography}
\end{document}